\documentclass[12pt]{article}
\usepackage{amsmath}
\usepackage{amssymb}
\usepackage{amsthm}
\usepackage{wrapfig}
\usepackage{braket}
\usepackage{mathrsfs}
\usepackage{cases}
\numberwithin{equation}{section}
\usepackage[top=26mm,bottom=26mm,left=24mm,right=24mm]{geometry}
\usepackage[abbrev]{amsrefs}

\newtheorem{df}{Definition}[section]
\newtheorem{lem}[df]{Lemma}
\newtheorem{prop}[df]{Proposition}

{\theoremstyle{definition}\newtheorem{rmk}[df]{Remark}}
\newtheorem{thm}[df]{Theorem}

\def\B{\mathbb{B}}

\def\R{\mathbb{R}}

\def\D{\mathscr{D}}

\def\O{\mathscr{O}}
\def\S{\mathscr{S}}

\def\div{{\rm{div}}\,}
\def\rot{{\rm{rot}}\,}

\begin{document}
\title{\vspace{-0.7cm}
\bf Existence of a Stationary Navier-Stokes Flow Past 
a Rigid Body, with Application to Starting Problem in Higher Dimensions}
\author{Tomoki Takahashi}
\date{}
\maketitle
\noindent{\bf Abstract.}~
We consider the large time behavior of the Navier-Stokes flow 
past a rigid body in $\R^n$ with $n\geq 3$. We first construct a small 
stationary solution possessing the optimal summability at spatial infinity, 
which is the same as that of the Oseen fundamental solution. 
When the translational velocity of the body gradually increases and 
is maintained after a certain finite time, 
we then show that 
the nonstationary fluid motion converges to 
the stationary solution corresponding to 
a small terminal velocity of the body 
as time $t\rightarrow\infty$ in $L^q$ with $q\in[n,\infty]$. 
This is called Finn's starting problem and the 
three-dimensional case was affirmatively 
solved by Galdi, Heywood and Shibata \cite{gahesh1997}. 
The present paper extends \cite{gahesh1997} 
to the case of higher dimensions. Even for the three-dimensional 
case, our theorem provides new convergence rate, that is determined by 
the summability of the stationary solution at infinity 
and seems to be sharp. 
\vspace{0.4cm} 
\vspace{0.4cm}
\section{Introduction and main results}\label{section1}
~~~We consider a viscous incompressible flow past a rigid body 
$\O\subset \R^n\,(n\geq 3)$. 
We suppose that 
$\O$ is translating 
with a velocity $-\psi(t)ae_1$, where $a>0$, 
$e_1=(1,0,\cdots,0)^\top$ and  
$\psi$ is a function on $\R$ describing the transition of the translational
velocity in such a way that 
\begin{align}\label{psidef}
\psi\in C^1(\R;\R),\quad
|\psi(t)|\leq 1\quad{\rm{for}}~\,t\in\R,\quad
\psi(t)=0\quad{\rm{for}}\,~t\leq 0,\quad 
\psi(t)=1\quad {\rm{for}} \,~t\geq 1. 
\end{align}
Here and hereafter, $(\cdot)^\top$ denotes the transpose.  
We take the frame attached to the body, 
then the fluid motion which occupies 
the exterior domain $D=\R^n\setminus\O$ 
with $C^2$ boundary $\partial D$ and is started from rest obeys 
\begin{align}\label{NS2}
\left\{
\begin{array}{r@{}c@{}ll}
\displaystyle\frac{\partial u}{\partial t}+u\cdot \nabla u&{}={}&
\Delta u-\psi(t)a\,\displaystyle\frac{\partial u}{\partial x_1}
-\nabla p,&\quad x\in D,~t>0,\\
\nabla\cdot u&{}={}&0,&\quad x\in D,t\geq 0,\\
u|_{\partial D}&{}={}&-\psi(t)ae_1,&\quad t>0,\\
u&{}\rightarrow{}&0&\quad {\rm{as}}~|x|\rightarrow \infty,\\
u(x,0)&{}={}&0,&\quad  x\in D.
\end{array}\right.
\end{align}
Here, $u=(u_1(x,t),\cdots,u_n(x,t))^\top$ and $p=p(x,t)$ denote 
unknown velocity and pressure of the fluid, respectively. 
Since $\psi(t)=1$ for $t\geq 1$, 
the large time behavior of solutions is  
related to the stationary problem 
\begin{align}\label{sta}
\left\{
\begin{array}{r@{}c@{}ll}
u_s\cdot \nabla u_s&{}={}&\Delta u_s-a\,
\displaystyle\frac{\partial u_s}{\partial x_1}-\nabla p_s,
&\quad x\in D,\\
\nabla\cdot u_s&{}={}&0,&\quad x\in D,\\
u_s|_{\partial D}&{}={}&-ae_1,\\
u_s&{}\rightarrow{}& 0&\quad {\rm{as}}~|x|\rightarrow \infty.
\end{array}\right.
\end{align}
When $n=3$, the pioneering work due to Leray \cite{leray1933} 
provided the existence theorem for 
weak solution to problem (\ref{sta}), what is called $D$-solution,  
having finite Dirichlet integral 
without smallness assumption on data. 
From the physical point of view, 
solutions of (\ref{sta}) should reflect the anisotropic decay structure 
caused by the translation, 
but his solution had little information about the 
behavior at large distances. To fill this gap, 
Finn \cite{finn1959}--\cite{finn1973} introduced the class of 
solutions with pointwise decay property 
\begin{align}\label{prclass}
u_s(x)=O(|x|^{-\frac{1}{2}-\varepsilon})
\quad\quad{\rm as~}|x|\rightarrow\infty
\end{align}
for some $\varepsilon>0$ 
and proved that if $a$ is small enough, (\ref{sta}) admits 
a unique solution satisfying (\ref{prclass}) and 
exhibiting paraboloidal wake region behind the body $\O$. 
He called the Navier-Stokes flows satisfying (\ref{prclass}) 
physically reasonable solutions. 
It is remarkable that 
$D$-solutions become physically reasonable solutions 
no matter how large $a$ would be, see Babenko \cite{babenko1973}, 
Galdi \cite{galdi1992} and Farwig and Sohr \cite{faso1998}. 
Galdi developed the $L^q$-theory of the linearized system, that 
we call the Oseen system, to prove that every $D$-solution 
has the same summability as the one of the Oseen fundamental solution 
without any smallness assumption, see \cite[Theorem X.6.4]{galdi2011}. 
It is not straightforward to generalize 
his result to the case of higher dimensions and 
it remains open whether the same result holds true for $n\geq 4$. 
We also refer to Farwig \cite{farwig1992t} who gave another outlook on 
Finn's results by using anisotropically weighted Sobolev spaces, 
and to Shibata \cite{shibata1999} who developed the estimates of 
physically reasonable solutions and then proved their stability 
in the $L^3$ framework when $a$ is small. 
There is less literature concerning the problem (\ref{sta}) 
for the case $n\geq 4$. 
When $n\geq 3$, 
Shibata and Yamazaki \cite{shya2005} 
constructed a solution $u_s$, which is 
uniformly bounded with respect to small $a\geq 0$ in the Lorentz space 
$L^{n,\infty}$, and 
investigated the behavior of $u_s$ as $a\rightarrow 0$. 
If, in particular, $n\geq 4$ and if $a\geq 0$ is sufficiently small, 
they also derived 
\begin{align}\label{staclass}
u_s\in L^{\frac{n}{1+\rho_1}}(D)\cap L^{\frac{n}{1-\rho_2}}(D),\quad 
\nabla u_s\in L^{\frac{n}{2+\rho_1}}(D)\cap L^{\frac{n}{2-\rho_2}}(D)
\end{align}
for some $0<\rho_1,\rho_2<1$, 
see \cite[Remark 4.2]{shya2005}. 
\par  Let us turn to the initial value problem. 
Finn \cite{finn1965s} conjectured that 
(\ref{NS2}) admits a solution which tends to 
a physically reasonable solution as $t\rightarrow\infty$ 
if $a$ is small enough. 
This is called Finn's starting problem. 
It was first studied 
by Heywood \cite{heywood1972}, in which a stationary solution is said to be 
attainable 
if the fluid motion converges to this solution as $t\rightarrow\infty.$ 
Later on, by using Kato's approach \cite{kato1984} 
(see also Fujita and Kato \cite{fuka1964}) together with the
$L^q$-$L^r$ estimates for the Oseen initial value problem
established by Kobayashi and Shibata \cite{kosh1998}, 
Finn's starting problem was affirmatively solved 
by Galdi, Heywood and Shibata \cite{gahesh1997}. After that, 
Hishida and Maremonti \cite{hima2018} 
constructed a sort of weak solution $u$ that enjoys  
\begin{align}\label{inftyconvergence}
\|u(t)-u_s\|_{\infty}=O(t^{-\frac{1}{2}})
 \quad {\rm as}~ t\rightarrow \infty
\end{align} 
if $a$ is small, but $u(\cdot,0)\in L^{3}(D)$ can be large. 
Here and hereafter, 
$\|\cdot\|_q$ denotes the norm of $L^q(D)$. 
Although we concentrate ourselves on attainability in this paper, 
stability of stationary solutions was also studied by 
Shibata \cite{shibata1999}, Enomoto and Shibata \cite{ensh2005} and 
Koba \cite{koba2017} in the $L^q$ framework. 
Those work except \cite{ensh2005} studied the 
three-dimensional exterior problem, while \cite{ensh2005} 
showed the stability of a stationary solution satisfying (\ref{staclass}) 
for some $0<\rho_1,\rho_2<1$ 
in $n$-dimensional exterior domains with $n\geq 3$. 
Stability of physically 
reasonable solutions in 2D is much more involved for several reasons and it 
has been recently proved by Maekawa \cite{maekawa2019stability}. 
\par The aim of this paper is two-fold. The first one is to construct a 
small stationary solution 
possessing the optimal summability at spatial infinity, 
which is the same as that of the Oseen fundamental solution $\bf E$:
\begin{align}\label{fundamentalint}
{\bf E}\in L^q(\{x\in\R^n\mid |x|>1\}),
\quad q>\frac{n+1}{n-1},\quad\quad
\nabla {\bf E}\in L^r(\{x\in\R^n\mid |x|>1\}),\quad
r>\frac{n+1}{n},
\end{align}
see Galdi \cite[Section VII]{galdi2011}. As already mentioned above, 
this result is well known in 
three-dimensional case even for large $a>0$, but it is not found 
in the literature for higher dimensional case $n\geq 4$. 
Our theorem covers the three-dimensional case as well 
and the proof is considerably shorter than the one given by 
authors mentioned above since we focus our interest 
only on summability at infinity 
rather than anisotropic pointwise estimates. 
The second aim is to give an affirmative answer 
to the starting problem as long 
as $a$ is small enough, that is, 
to show the attainability of the stationary 
solution obtained above. The result extends 
Galdi, Heywood and Shibata \cite{gahesh1997} to the case 
of higher dimensions. Even for the three-dimensional case, our theorem 
not only recovers \cite{gahesh1997} 
but also provides better decay properties, for instance, 
\begin{align}\label{inftyconvergence2}
\|u(t)-u_s\|_{\infty}=O(t^{-\frac{1}{2}-\frac{\rho}{2}})
 \quad {\rm as}~ t\rightarrow \infty
\end{align} 
for some $\rho>0$, that should be compared with (\ref{inftyconvergence}). 
This is because the fluid is initially at rest and because the 
three-dimensional stationary solution $u_s$ belongs to $L^q(D)$ 
with $q<3$; to be precise, since $q$ can be close to $2$, one can take 
$\rho$ close to $1/2$ in (\ref{inftyconvergence2}).
Due to the $L^q$-$L^r$ estimates of the Oseen semigroup 
established by Kobayashi and Shibata \cite{kosh1998}, 
Enomoto and Shibata \cite{ensh2004,ensh2005}, 
see Proposition \ref{proplqlr},
this decay rate is sharp in view 
of presence of $u_s$, see (\ref{h1}) below, in forcing terms of 
the equation (\ref{NS3}) for the perturbation. 
Our result can be also compared with \cite{tomopre} 
by the present author on the 
starting problem in 
which translation is replaced by rotation of the body $\O\subset \R^3$. 
Under the circumstance of \cite{tomopre}, the optimal spatial decay of 
stationary solutions observed in general is the scale-critical rate 
$O(|x|^{-1})$, so that they cannot belong to $L^q(D)$ with $q\leq 3=n$, 
and therefore, we have no chance to deduce (\ref{inftyconvergence2}). 
Another remark is that, in comparison with stability theorem due to 
\cite{ensh2005} for $n\geq 3$, more properties 
of stationary solutions are needed to establish the attainability theorem. 
Therefore, those properties must be deduced in constructing a solution of 
(\ref{sta}). 
\par Let us state the first main theorem on 
the existence and summability of stationary solutions.
\begin{thm}\label{rmkrho}
Let $n\geq 3$. For every $(\alpha_1,\alpha_2,\beta_1,\beta_2)$ satisfying 
\begin{align}\label{parameter1}
\frac{n+1}{n-1}<\alpha_1\leq n+1\leq \alpha_2<\frac{n(n+1)}{2},\quad 
\frac{n+1}{n}<\beta_1\leq \frac{n+1}{2}\leq \beta_2<\frac{n(n+1)}{n+2},
\end{align}
there exists a constant  
$\delta=\delta(\alpha_1,\alpha_2,\beta_1,\beta_2,n,D)\in(0,1)$ such 
that if
\begin{align*}
0<a^{\frac{n-2}{n+1}}<\delta,
\end{align*}
problem $(\ref{sta})$ admits a unique solution $u_s$ along with 
\begin{align}\label{usest}
\|u_s\|_{\alpha_1}+\|u_s\|_{\alpha_2}\leq Ca^{\frac{n-1}{n+1}},\quad 
\|\nabla u_s\|_{\beta_1}+\|\nabla u_s\|_{\beta_2}
\leq Ca^{\frac{n}{n+1}},
\end{align}
where $C>0$ is independent of $a$. 
\end{thm}
\par The upper bounds of $\alpha_2$ and $\beta_2$ come from (\ref{s})
with $q<n/2$ in Proposition \ref{proplinear} on 
the $L^q$-theory of the Oseen system, 
whereas the lower 
bounds of $\alpha_1$ and $\beta_1$ are just (\ref{fundamentalint}). 
\par For the proof of Theorem \ref{rmkrho}, 
we define a certain closed ball 
$N$ and a contraction map 
$\Psi:N\ni v\mapsto u\in N$ 
which provides the solution to the problem 
\begin{align}\label{vnablav}
\left\{
\begin{array}{r@{}c@{}ll}
\Delta u-a\,\displaystyle\frac{\partial u}{\partial x_1}
&{}={}&\nabla p+v\cdot\nabla v,
&\quad x\in D,\\
\nabla\cdot u&{}={}&0,&\quad x\in D,\\
u|_{\partial D}&{}={}&-ae_1,\\
u&{}\rightarrow{}& 0&\quad {\rm{as}}~|x|\rightarrow \infty.
\end{array}\right.
\end{align}
In doing so, we rely on the $L^q$-theory of the Oseen system developed by 
Galdi \cite[Theorem VII.7.1]{galdi2011}, see Proposition \ref{proplinear}, 
which gives us sharp summability estimates of solutions 
at infinity together with explicit dependence on $a>0$. 
As long as we only use Proposition \ref{proplinear}, 
the only space in which estimates of $\Psi$ are closed is 
\begin{align*}
\{u\in L^{n+1}(D)\mid\nabla u\in L^{\frac{n+1}{2}}(D)\}.
\end{align*}
From this, we can capture neither the optimal summability 
at infinity nor regularity 
required in the study of the starting problem. 
We thus use at least two spaces 
$L^{\alpha_i}(D)~(i=1,2)$ for $u$ 
and $L^{\beta_i}(D)~(i=1,2)$ for $\nabla u$, 
and intend to find a solution within a closed ball $N$ of 
\begin{align}\label{staclass3}
\{u\in L^{\alpha_1}(D)\cap L^{\alpha_2}(D)\mid  
\nabla u\in L^{\beta_1}(D)\cap L^{\beta_2}(D)\}.
\end{align} 
However, it is not possible to apply Proposition \ref{proplinear} to 
$f=v\cdot\nabla v$ with 
\begin{align}\label{alpha1beta1}
v\in L^{\alpha_1}(D),\quad\quad \nabla v\in L^{\beta_1}(D)
\end{align}
or
\begin{align}\label{alpha2beta2}
v\in L^{\alpha_2}(D),\quad\quad \nabla v\in L^{\beta_2}(D)
\end{align}
if $\alpha_1$ and $\beta_1$ are simultaneously close to 
$(n+1)/(n-1)$ and $(n+1)/n$, or 
if $\alpha_2$ and $\beta_2$ are simultaneously close to 
$n(n+1)/2$ and $n(n+1)/(n+2)$, because 
the relation 
\begin{align*}
\frac{2}{n}<\frac{1}{\alpha_2}+\frac{1}{\beta_2}<
\frac{1}{\alpha_1}+\frac{1}{\beta_1}<1
\end{align*}
required in the linear theory, see Proposition \ref{proplinear}, 
is not satisfied. 
In order to overcome this difficulty, given 
$(\alpha_1,\alpha_2,\beta_1,\beta_2)$ satisfying (\ref{parameter1}), 
we choose auxiliary exponents $(q_1,q_2,r_1,r_2)$ fulfilling 
\begin{align*}
\alpha_1\leq q_1\leq q_2\leq \alpha_2,
\quad\beta_1\leq r_1\leq r_2\leq\beta_2,
\quad \frac{2}{n}<\frac{1}{q_i}+\frac{1}{r_i}<1,\quad i=1,2
\end{align*}
such that the application of Proposition \ref{proplinear} 
to $f=v\cdot\nabla v$ with 
$v\in L^{q_1}(D)$ and $\nabla v\in L^{r_1}(D)$ 
(resp. $v\in L^{q_2}(D)$ and $\nabla v\in L^{r_2}(D)$) 
recovers (\ref{alpha1beta1}) (resp. (\ref{alpha2beta2})) with $u$. 
\par Another possibility to prove 
Theorem \ref{rmkrho} is combining Proposition \ref{proplinear} with 
the Sobolev inequality. 
We then get a solution $(u_s,p_s)\in X_q(n)$ for all $q\in(1,\infty)$ 
with $n/3\leq q\leq (n+1)/3$, where $X_q(n)$ is defined in 
Proposition \ref{proplinear}. 
The restriction $n/3\leq q\leq (n+1)/3$ is removed by  
applying a bootstrap argument
to decrease the lower bound to $1$ 
and to increase the upper bound to $n/2$. 
As compared with this way, in our proof, 
we do not any use a bootstrap argument and directly construct a
solution possessing the optimal summability at infinity as well as 
regularity required in the study of the starting problem. 
\par Let us proceed to the starting problem. 
To study the attainability of the stationary solution 
$u_s$ of class (\ref{staclass3}) with $(\alpha_1,\alpha_2,\beta_1,\beta_2)$
satisfying (\ref{parameter1}), 
it is convenient to set 
\begin{align}\label{staclass2}
\alpha_1=\frac{n}{1+\rho_1},\quad\quad
\alpha_2=\frac{n}{1-\rho_2},\quad\quad 
\beta_1=\frac{n}{2+\rho_3},\quad\quad\beta_2=\frac{n}{2-\rho_4}
\end{align}
with $(\rho_1,\rho_2,\rho_3,\rho_4)$ satisfying
\begin{align}\label{rho1}
0<\rho_1<\frac{n^2-2n-1}{n+1},\quad
\frac{1}{n+1}\leq \rho_2<\frac{n-1}{n+1},\quad 
0<\rho_3<\frac{n^2-2n-2}{n+1},\quad
\frac{2}{n+1}\leq \rho_4<\frac{n}{n+1}
\end{align}
and we need the additional condition 
\begin{align}\label{rho2}
\rho_2+\rho_4>1.
\end{align}
We note that the set of those parameters is nonvoid. 
It is reasonable to look for a solution to (\ref{NS2}) of the form 
\begin{align*}
u(x,t)=v(x,t)+\psi(t)u_s,\quad p(x,t)=\phi(x,t)+\psi(t)p_s.
\end{align*}
Then the perturbation $(v,\phi)$ 
satisfies the following initial boundary value problem
\begin{align}\label{NS3}\left\{
\begin{array}{r@{}c@{}l}
\displaystyle\frac{\partial v}{\partial t}&{}={}&\Delta v-a\displaystyle
\frac{\partial v}{\partial x_1}
-v\cdot\nabla v-\psi(t)v\cdot\nabla u_s
-\psi(t)u_s\cdot\nabla v
+(1-\psi(t))a\frac{\partial v}{\partial x_1}\\
&&\hspace{5cm}+h_1(x,t)+h_2(x,t)
-\nabla \phi,
\quad x\in D,\,t>0,\\
\nabla\cdot v&{}={}&0,\quad x\in D,\,t\geq 0,\\
v|_{\partial D}&{}={}&0,\quad t>0,\\
v&{}\rightarrow{}& 0~\quad {\rm{as}}~|x|\rightarrow \infty,\\
v(x,0)&{}={}&0,\quad x\in D,
\end{array}\right.
\end{align}
where 
\begin{align}
&h_1(x,t)=-\psi'(t)u_s,\label{h1}\\
&h_2(x,t)=\psi(t)\big(1-\psi(t)\big)
\Big(u_s\cdot\nabla u_s+a\frac{\partial u_s}{\partial x_1}\Big).\label{h2}
\end{align}
In what follows, we study the problem (\ref{NS3}) instead of (\ref{NS2}). 
In fact, if we obtain the solution $v$ of (\ref{NS3}) 
which converges to $0$ as $t\rightarrow \infty$, 
the solution $u$ of (\ref{NS2}) converges to $u_s$ as 
$t\rightarrow \infty$. 
Problem (\ref{NS3}) is converted into
\begin{align}\label{NS4}
v(t)=\int_0^t e^{-(t-\tau)A_a}P\Big[-v\cdot\nabla v
&-\psi(\tau)v\cdot\nabla u_s-\psi(\tau)u_s\cdot\nabla v\nonumber\\
&+\big(1-\psi(\tau)\big)a\frac{\partial v}{\partial x_1}
+h_1(\tau)+h_2(\tau)\Big]d\tau
\end{align}
by using the Oseen semigroup $e^{-tA_a}$ (see Section 3) as well as 
the Fujita-Kato projection $P$ 
from $L^q(D)$ onto $L^q_{\sigma}(D)$ 
associated with the Helmholtz decomposition 
(see Fujiwara and Morimoto \cite{fumo1977}, 
Miyakawa \cite{miyakawa1982} and Simader and Sohr \cite{siso1992}): 
\begin{align*}
L^q(D)=L^q_{\sigma}(D)\oplus\{\nabla p\in L^q(D)\mid
p\in L^q_{\rm{loc}}(\overline{D})\}\quad (1<q<\infty).
\end{align*}
Here,
\begin{align*}
L^q_{\sigma}(D)=\overline{C^\infty_{0,\sigma}(D)}^{\|\cdot\|_q}, \quad 
C^\infty_{0,\sigma}(D)=\{u\in C^\infty_0(D)^n\mid\nabla\cdot u=0\}.
\end{align*}

We are now in a position to give the second main theorem on 
attainability of stationary solutions. 
\begin{thm}\label{attainthm}
Let $n\geq 3$ and let 
$\psi$ be a function on $\R$ satisfying $(\ref{psidef}).$ 
We set $M=\displaystyle\max_{t\in\R}|\psi'(t)|$. Suppose that 
$\rho_1$, $\rho_2$, $\rho_3$ and $\rho_4$ satisfy 
$(\ref{rho1})$--\,$(\ref{rho2})$ and let $\delta$
be the constant in Theorem \ref{rmkrho} with $(\ref{staclass2}).$ 
Then there exists a constant 
$\varepsilon=\varepsilon(n,D)\in(0,\delta]$ 
such that if
\begin{align*}
0<(M+1)a^{\frac{n-2}{n+1}}<\varepsilon,
\end{align*} 
$(\ref{NS4})$ admits a unique solution $v$ within the class
\begin{align}\label{yr0}
Y_{0}:=\big\{v\in BC([0,\infty);L^n_{\sigma}(D))\mid
t^{\frac{1}{2}}v&\in BC((0,\infty);L^{\infty}(D)),
t^{\frac{1}{2}}\nabla v\in BC((0,\infty);L^n(D)),\nonumber\\
\lim_{t\rightarrow 0}&~t^{\frac{1}{2}}\big(
\|v(t)\|_{\infty}+\|\nabla v(t)\|_{n}\big)=0\big\}.
\end{align}
\par Moreover, we have the following.
\begin{enumerate}
\item $(sharp~decay)$\quad Let $n=3$.
Then there exists a constant 
$\varepsilon_*=\varepsilon_*(D)\in(0,\varepsilon]$ 
such that if \,$0<(M+1)a^{1/4}<\varepsilon_*$, 
the solution $v$ 
enjoys decay properties  
\begin{align}
&\|v(t)\|_q=O(t^{-\frac{1}{2}+\frac{3}{2q}-\frac{\rho_1}{2}}),\quad
\quad 3\leq \forall q\leq \infty \label{sharpdecay},\\
&\|\nabla v(t)\|_3=
O(t^{-\frac{1}{2}-\frac{\rho_1}{2}})\label{sharpdecaygrad}
\end{align}
as $t\rightarrow \infty$.\\
\quad Let $n\geq4$ and suppose that $\rho_3>1$ and 
$1<\rho_1\leq 1+\rho_3$ in 
addition to $(\ref{rho1})$ 
$($the set of those parameters is nonvoid when $n\geq 4$\,$)$. 
Then there exists a constant 
$\varepsilon_*=\varepsilon_*(n,D)\in(0,\varepsilon]$ 
such that if\, 
$0<(M+1)a^{(n-2)/(n+1)}<\varepsilon_*$, the solution $v$ 
enjoys 
\begin{align}
&\|v(t)\|_q=O(t^{-\frac{1}{2}+\frac{n}{2q}-\frac{\rho_1}{2}}),\quad
\quad n\leq \forall q\leq \infty \label{sharpdecayn},\\
&\|\nabla v(t)\|_n=
O(t^{-\frac{1}{2}-\frac{\rho_1}{2}})\label{sharpdecaygradn}
\end{align}
as $t\rightarrow\infty.$
\item $(Uniqueness)$\quad There exists a constant 
$\hat{\varepsilon}=\hat{\varepsilon}(n,D)\in(0,\varepsilon]$ 
such that if \,$0<(M+1)a^{(n-2)/(n+1)}<\hat{\varepsilon}$, 
the solution $v$ obtained above is unique even within the class 
\begin{align}\label{yr}
Y:=\{v\in BC([0,\infty);L^n_{\sigma}(D))\mid 
t^{\frac{1}{2}}v\in BC((0,\infty);L^{\infty}(D)),
t^{\frac{1}{2}}\nabla v\in BC((0,\infty);L^n(D))\}.
\end{align}
\end{enumerate}
\end{thm}
\par For the sharp decay properties 
(\ref{sharpdecay})--(\ref{sharpdecaygradn}),  
the key step is to prove the $L^n$-decay of the solution, that is, 
\begin{align}\label{ninfty}
\|v(t)\|_n=O(t^{-\frac{\rho_1}{2}})
\end{align}
as $t\rightarrow\infty$. 
Once we have (\ref{ninfty}), the other decay properties can be derived 
by the similar argument to \cite{ensh2005}. 
Note that the condition 
$\rho_1\leq 1+\rho_3$ is always fulfilled and thus redundant 
for $n=3$ since $\rho_1<1/2$ and $\rho_3<1/4$. 
On the other hand, it is enough for $n\geq 4$ to consider the case 
$\rho_1,\rho_3>1$. 
To prove (\ref{ninfty}), we first derive slower decay
\begin{align*}
\|v(t)\|_n=O(t^{-\frac{\rho}{2}})
\end{align*}
with some $\rho\in(0,1)$ by making use of 
$u_s\in L^{n/(1+\rho_1)}(D)$ and $\nabla u_s\in L^{n/(2+\rho_3)}(D)$, 
see Lemma \ref{propsharpdecay} in Section 3. When $n=3$, one can take 
$\rho:=\min\{\rho_1,\rho_3\}$, yielding better decay properties 
of the other norms of the solution. 
With them at hand, we repeat improvement of the 
estimate of $\|v(t)\|_n$ step by step to find (\ref{ninfty}).
However, this procedure does not work for $n\geq 4$ because of $\rho_1>1$. 
In order to get around the difficulty, our idea is to deduce 
the $L^{q_0}$-decay of the solution with some $q_0<n$, that is 
appropriately chosen, see Lemma \ref{lemsharp}. 
We are then able to repeat improvement of estimates of several terms to 
arrive at (\ref{ninfty}), where the argument is more involved 
than the three-dimensional case above. 
Finally, to prove the uniqueness within $Y$, 
we employ the idea developed by Brezis \cite{brezis1994}, which 
shows that the solution $v\in Y$ necessarily satisfies the behavior as 
$t\rightarrow 0$ in (\ref{yr0}). 
\par In the next section we introduce the $L^q$-theory 
of the Oseen system and then prove Theorem \ref{rmkrho}. 
The final section is devoted to the proof of Theorem \ref{attainthm}.

\section{Proof of Theorem \ref{rmkrho}}
~~~In order to prove Theorem \ref{rmkrho}, we first recall the result on 
the Oseen boundary value problem due to 
Galdi \cite[Theorem VII.7.1]{galdi2011}, 
see also Galdi \cite{galdi1992oseen} for the first proof of this result. 
\begin{prop}\label{proplinear}
Let $n\geq3$ and let $D\subset\R^n$ be an exterior domain with 
$C^2$ boundary. Suppose $a>0$ and $1<q<(n+1)/2$. 
Given $f\in L^q(D)$ and $u_*\in W^{2-1/q,q}(\partial D)$, problem  
\begin{align}\label{stalinear}
\left\{
\begin{array}{r@{}c@{}ll}
\Delta u-a\,\displaystyle\frac{\partial u}{\partial x_1}
&{}={}&\nabla p+f,
&\quad x\in D,\\
\nabla\cdot u&{}={}&0,&\quad x\in D,\\
u|_{\partial D}&{}={}&u_*,\\
u&{}\rightarrow{}& 0&\quad {\rm{as}}~|x|\rightarrow \infty
\end{array}\right.
\end{align}
admits a unique $($up to an additive constant for $p)$ solution $(u,p)$ 
within the class
\begin{align*}
X_q(n):=\Big\{(u,p)\in L^1_{\rm{loc}}(D)\,\Big|\,
u\in L^{s_2}(D),\,&\nabla u\in L^{s_1}(D),\,\nabla^2u\in L^{q}(D),\nonumber
\\&
\displaystyle\frac{\partial u}{\partial x_1}\in 
L^q(D),\,\nabla p\in L^{q}(D)\Big\},
\end{align*}
where 
\begin{align}\label{s}
\frac{1}{s_1}=\frac{1}{q}-\frac{1}{n+1},\quad\frac{1}{s_2}=\frac{1}{q}
-\frac{2}{n+1}. 
\end{align}
Here, by 
$W^{2-1/q,q}(\partial D)$ 
we denote the trace space on $\partial D$ 
from the Sobolev space $W^{2,q}(D)$ 
$($see, for instance, \cite{adams1975} and \cite{galdi2011}$)$. 
\par If, in particular, $a\in (0,1]$ and $q<n/2$, 
then the solution $(u,p)$ obtained above satisfies 
\begin{align*}
a^{\frac{2}{n+1}}\|u\|_{s_2}+
a\left\|\displaystyle\frac{\partial u}{\partial x_1}\right\|_q+
a^{\frac{1}{n+1}}\|\nabla u\|_{s_1}
+\|\nabla^2 u\|_{q}+\|\nabla p\|_{q}\leq
C\big(\|f\|_q+\|u_*\|_{W^{2-\frac{1}{q},q}(\partial D)}\,\big)
\end{align*}
with a constant $C>0$ dependent on $q,n$ and $D$, 
however, independent of $a$.
\end{prop}
For later use, we prepare the following lemma. 
The proof is essentially same as  
the one of Young's inequality for convolution, thus we omit it.   
\begin{lem}\label{young}
Let $R_0,\,d>0$. Assume that $1\leq q,s\leq\infty$ and 
$1/q+1/s\geq 1$. Suppose $u\in L^q(\R^n)$ 
with ${\rm supp}\,u\subset B_d:=\{x\in\R^n\mid |x|<d\}$ 
and $\rho\in L^s(\R^n\setminus B_{R_0})$. 
Then for all 
$R\geq R_0+d$, $\rho*u$ is well-defined as an element of 
$\,L^r(\R^n\setminus B_R)$ together with 
\begin{align*}
\|\rho*u\|_{L^r(\R^n\setminus B_{R})}\leq 
\|\rho\|_{L^s(\R^n\setminus B_{R_0})}\|u\|_{L^q(B_d)},
\end{align*}
where $*$ denotes the convolution and $1/r:=1/q+1/s-1$.
\end{lem}
When the external force $f$ is taken from $L^{q_1}(D)\cap L^{q_2}(D)$ with 
$1<q_1,q_2<(n+1)/2$ and $q_1\ne q_2$, we can 
apply Proposition \ref{stalinear} to $f\in L^{q_i}(D)$ $(i=1,2)$. 
The following tells us that 
the corresponding solutions coincide with each other. 
\begin{lem}\label{identify}
Suppose $n\geq 3$, 
$1<q_1,q_2<(n+1)/2$ and $f\in L^{q_1}(D)\cap L^{q_2}(D)$. Let 
$(u_i,p_i)$ be a unique solution obtained in 
Proposition \ref{stalinear} with $f\in L^{q_i}(D)$ and $u_*=-ae_1$. 
Then $u_1=u_2$.
\end{lem}
\begin{proof}
We first show that $u_1-u_2$ behaves like the Oseen fundamental solution 
$\bf E$ at large distances. 
We fix $R_0>0$ satisfying 
$\R^n\setminus D\subset B_{R_0}$. 
Let $\zeta\in C^\infty(\R^n)$ be a cut-off function 
such that $\zeta(x)=0$ for $|x|\leq R_0$, 
$\zeta(x)=1$ for $|x|\geq R_0+1$, and set 
\begin{alignat*}{2}
u(x)&:=u_1(x)-u_2(x),&\quad\quad p(x)&:=p_1(x)-p_2(x),
\\
 v(x)&:=\zeta(x) u(x)-\B[u\cdot \nabla\zeta],&\pi(x)&:=\zeta(x)p(x).
\end{alignat*}
Here, $\B$ is the Bogovski\u{\i} operator defined on the domain 
$B_{R_0+1}\setminus B_{R_0}$, see Bogovski\u{\i} \cite{bogovskii1979}, 
Borchers and Sohr \cite{bosh1990} and Galdi \cite{galdi2011}. 
Then we have 
\begin{align}\label{temp}
-\Delta v+a\frac{\partial v}{\partial x_1}+\nabla \pi=g(x),\quad 
\nabla\cdot v=0\quad {\rm{in}}~\S'(\R^n),
\end{align}
where $\S'(\R^n)$ is the set of tempered distributions on $\R^n$ and 
\begin{align*}
g(x)=-(\Delta \zeta)u-2(\nabla \zeta\cdot\nabla)u
+a\frac{\partial \zeta}{\partial x_1}u+p\nabla \zeta+
\Big(\Delta-a\frac{\partial}{\partial x_1}\Big)\B[u\cdot \nabla\zeta].
\end{align*}
For (\ref{temp}) with $g=0$, we have supp\,$\hat{v}\subset \{0\}$ and 
supp\,$\hat{\pi}\subset \{0\}$,
where $\hat{(\cdot)}$ denotes the Fourier transform. We thus find 
\begin{align*}
v(x)=\int_{\R^n} {\bf{E}}(x-y)g(y)\,dy+P(x),\quad\quad
\pi(x)=C(n)\int_{\R^n} \frac{x-y}{|x-y|^n}\cdot g(y)\,dy+Q(x)
\end{align*}
with some polynomials $P(x)$, $Q(x)$ and some constant $C(n)$. 
In view of 
$v\in L^{(\,1/q_1-2/(n+1)\,)^{-1}}(\R^n)$
$+L^{(\,1/q_2-2/(n+1)\,)^{-1}}(\R^n)$ 
and $\nabla \pi\in L^{q_1}(\R^n)+L^{q_2}(\R^n)$, we have 
$P(x)=0$ and $Q(x)=\overline{p}$. Here,  
$\overline{p}$ is some constant.  
Then Lemma \ref{young} 
with 
\begin{align*}
\rho={\bf E},~\nabla {\bf E},~\frac{x-y}{|x-y|^n}, 
\end{align*}
$u=g$, $d=R_0+1$, 
$q=1$ and $r=s$ leads us to 
\begin{align}\label{uint}
&u\in L^q(\R^n\setminus B_{2R_0+1}),\quad\quad
\nabla u\in L^r(\R^n\setminus B_{2R_0+1}),\quad
\quad p-\overline{p}\in L^s(\R^n\setminus B_{2R_0+1})
\end{align}  
for all $q>(n+1)/(n-1)$, $r>(n+1)/n$ and $s>n/(n-1)$, 
see (\ref{fundamentalint}). 
\par Let $\varphi\in C^{\infty}[0,\infty)$ be a cut-off function such that 
$\varphi(t)=1$ for $t\leq 1$, $\varphi(t)=0$ for $t\geq 2$, and set 
$\varphi_R(x)=\varphi(|x|/R)$ for $R\geq 2R_0+1$, 
$x\in\R^n$. We note 
that there exists a constant $C>0$ independent of $R$ such that 
\begin{align}\label{independentc}
\|\nabla \varphi_R\|_n\leq C.
\end{align}
It follows from 
\begin{align*}
-\Delta u+a\frac{\partial u}{\partial x_1}+\nabla p=0,\quad 
\nabla\cdot u=0\quad {\rm{in}}~D,\quad u|_{\partial D}=0
\end{align*}
that 
\begin{align}
0&=\int_D\Big\{-\Delta u
+a\frac{\partial u}{\partial x_1}
+\nabla (p-\overline{p})\Big\}\cdot (\varphi_R u)\,dx\nonumber\\
&=\int_D|\nabla u|^2\varphi_R\,dx
+\int_{R\leq|x|\leq 2R}\Big\{(\nabla u\cdot\nabla\varphi_R)u-\frac{a}{2}
\frac{\partial \varphi_R}{\partial x_1}|u|^2-
(p-\overline{p})\nabla \varphi_R\cdot u\Big\}\,dx.\label{ur}
\end{align} 
Since we can see
\begin{align*} 
|\nabla u||u|,\,|u|^2,\,(p-\overline{p})|u|\in L^{n/(n-1)}
\big(\R^n\setminus B_{2R_0+1}) 
\end{align*}
from (\ref{uint}), 
letting $R\rightarrow \infty$ in (\ref{ur}) yields 
$\|\nabla u\|^2_2=0$ because of (\ref{independentc}). 
From this together with 
$u|_{\partial D}=0$, we conclude $u_1=u_2$. 
\end{proof}
\noindent {\bf Proof of Theorem \ref{rmkrho}}\quad Let $n\geq 3$ and  
let $(\alpha_1,\alpha_2,\beta_1,\beta_2)$ satisfy (\ref{parameter1}). 
We first choose parameters $(q_1,q_2,r_1,r_2)$ satisfying 
\begin{align}
&\frac{n+1}{n-1}<\alpha_1\leq q_1\leq n+1
\leq q_2\leq \alpha_2<\frac{n(n+1)}{2},\label{parameter2}\\
&\frac{n+1}{n}<\beta_1\leq r_1\leq \frac{n+1}{2}
\leq r_2\leq \beta_2<\frac{n(n+1)}{n+2},\label{parameter3}\\
&\max \left\{\frac{1}{\alpha_1}+\frac{2}{n+1},
\frac{1}{\beta_1}+\frac{1}{n+1}\right\}\leq 
\frac{1}{q_1}+\frac{1}{r_1}<1,\label{parameter4}\\
&\frac{2}{n}<\frac{1}{q_2}+\frac{1}{r_2}\leq 
\min\left\{\frac{1}{\alpha_2}+\frac{2}{n+1},
\frac{1}{\beta_2}+\frac{1}{n+1}\right\}.\label{parameter5}
\end{align}
It is actually possible to choose those parameters. In fact, we put 
\begin{align*}
\alpha_1=\frac{n+1}{n-1-\gamma_1},\quad 
\alpha_2=\frac{n(n+1)}{2+\gamma_2},\quad 
\beta_1=\frac{n+1}{n-\eta_1},\quad \beta_2=\frac{n(n+1)}{n+2+\eta_2}
\end{align*}
with arbitrarily small $\gamma_i,\eta_i\in (0,n-2]$ and 
look for $(q_1,q_2,r_1,r_2)$ of the form 
\begin{align*}
q_1=\frac{n+1}{n-1-\tilde{\gamma}_1},\quad
 q_2=\frac{n(n+1)}{2+\tilde{\gamma}_2},\quad 
r_1=\frac{n+1}{n-\tilde{\eta}_1},\quad 
r_2=\frac{n(n+1)}{n+2+\tilde{\eta}_2}.
\end{align*}
Then the conditions (\ref{parameter2})--(\ref{parameter5}) 
are accomplished by  
\begin{align*}
&n-2<\tilde{\gamma}_1+\tilde{\eta}_1\leq n-2+ \min\{\gamma_1,\eta_1\},\quad
\quad 
n-2<\tilde{\gamma}_2+\tilde{\eta}_2\leq n-2+\min\{\gamma_2,\eta_2\},\\ 
&\gamma_i\leq \tilde{\gamma}_i,\quad\quad \eta_i\leq\tilde{\eta}_i,
\quad i=1,2.
\end{align*}
For each $i=1,2$, the set of $(\tilde{\gamma}_i,\tilde{\eta}_i)$ with those 
conditions is nonvoid for given $\gamma_i$ and $\eta_i$; for instance, 
we may take $\tilde{\gamma}_i=\gamma_i$, 
$\tilde{\eta}_i=n-2$ when $\gamma_i\leq \eta_i$ and 
take $\tilde{\gamma}_i=n-2$, 
$\tilde{\eta}_i=\eta_i$ when $\gamma_i\geq \eta_i$. 
\par To obtain a small solution, we use the contraction mapping principle. 
We define  
\begin{align*}
B:=\{u\in L^{\alpha_1}(D)\cap L^{\alpha_2}(D) \mid 
\nabla u\in L^{\beta_1}(D)\cap L^{\beta_1}(D)\}
\end{align*}
which is a Banach space
endowed with the norm
\begin{align*}
\|u\|_B:=\sum^2_{i=1}(a^\frac{2}{n+1}\|u\|_{\alpha_i}
+a^\frac{1}{n+1}\|\nabla u\|_{\beta_i}).
\end{align*} 
Given $v\in B$, which satisfies  
\begin{align*}
v\cdot \nabla v\in \bigcap_{i=1}^2L^{\kappa_i}(D),\quad\quad\quad 
\frac{1}{\kappa_i}=\frac{1}{q_i}+\frac{1}{r_i},
\quad 1<\kappa_i<\frac{n}{2}
\end{align*}
for $i=1,2,$
we can employ Proposition \ref{stalinear} 
with $f=v\cdot \nabla v$, $q=\kappa_i$ $(i=1,2)$ and $u_*=-ae_1$. 
Then, due to Lemma \ref{identify}, the problem 
(\ref{vnablav}) admits a unique solution $(u,p)$ such that 
\begin{align*}
a^{\frac{2}{n+1}}\|u\|_{\mu_i}&+
a\left\|\frac{\partial u}{\partial x_1}\right\|_{\kappa_i}+
a^{\frac{1}{n+1}}\|\nabla u\|_{\lambda_i}+\|\nabla^2 u\|_{\kappa_i}+
\|\nabla p\|_{\kappa_i}\nonumber\\&\leq 
C'(\|v\cdot \nabla v\|_{\kappa_i}+a)
\leq C'(\|v\|_{q_i}\|\nabla v\|_{r_i}+a)
\leq C'(a^{-\frac{3}{n+1}}\|v\|_B^2+a)
\end{align*}
for $i=1,2.$ Here, 
$1/\lambda_i=1/\kappa_i-1/(n+1),$ $1/\mu_i=1/\kappa_i-2/(n+1)$. 
Furthermore, 
because the conditions (\ref{parameter4}) and (\ref{parameter5}) ensure 
$\mu_1\leq \alpha_1\leq\alpha_2\leq \mu_2$ and 
$\lambda_1\leq \beta_1\leq \beta_2\leq \lambda_2$, 
we find $u\in B$ with 
\begin{align*}
\|u\|_B\leq 4C'(a^{-\frac{3}{n+1}}\|v\|_B^2+a).
\end{align*}
Hence, we assume 
\begin{align}\label{acondi}
a^{\frac{n-2}{n+1}}<\frac{1}{64C'^2}=:\delta
\end{align}
and set 
\begin{align*}
N_a:=\{u\in B\mid \|u\|_B\leq 8C'a\}
\end{align*}
to see that the map 
$\Psi:N_a\ni v\mapsto u\in N_a$ 
is well-defined. Moreover, for $v_i\in N_a~(i=1,2)$, 
set $u_i=\Psi(v_i)$ and let $p_i$ be the  
pressure associated with $u_i$. Then we have
\begin{align*}
\left\{
\begin{array}{r@{}c@{}ll}
\Delta (u_1-u_2)-a\,
\displaystyle\frac{\partial}{\partial x_1}(u_1-u_2)
&{}={}&\nabla(p_1-p_2)+(v_1-v_2)\cdot\nabla v_1+
v_2\cdot\nabla(v_2-v_1),
&\quad x\in D,\\
\nabla\cdot(u_1-u_2)&{}={}&0,\hspace{1cm}x\in D,\\
(u_1-u_2)|_{\partial D}&{}={}&0,\\
u_1-u_2&{}\rightarrow{}& 0\hspace{1.2cm}{\rm{as}}~|x|\rightarrow \infty.
\end{array}\right.
\end{align*}
By applying Proposition \ref{stalinear} again, we find 
\begin{align*}
\|u_1-u_2\|_B\leq
4C'a^{-\frac{3}{n+1}}(\|v_1\|_B+\|v_2\|_B)\|v_1-v_2\|_B
\leq 64C'^2a^{\frac{n-2}{n+1}}\|v_1-v_2\|_B
\end{align*}
and the map $\Psi$ is contractive on
account of (\ref{acondi}). The proof is complete. \qed

\section{Proof of Theorem \ref{attainthm}}
~~~In this section, we prove Theorem \ref{attainthm}. 
We define the operator 
$A_a:L^q_{\sigma}(D)\rightarrow L^{q}_{\sigma}(D)
~(a>0,1<q<\infty)$ by 
\begin{align*}
\D(A_a)=W^{2,q}(D)\cap W^{1,q}_0(D)\cap L^{q}_{\sigma}(D),\quad 
A_au=-P\left[\Delta u-a\frac{\partial u}{\partial x_1}\right].
\end{align*}
Here, $W_0^{1,q}(D)$ denotes 
the completion of $C_0^\infty(D)$ in the Sobolev space $W^{1,q}(D)$. 
It is well known that $-A_a$ generates an analytic $C_0$-semigroup 
$e^{-tA_a}$ called the Oseen semigroup in $L^q_{\sigma}(D)$, 
see Miyakawa \cite[Theorem 4.2]{miyakawa1982}, 
Enomoto and Shibata \cite[Theorem 4.4]{ensh2004}. 
The following $L^q$-$L^r$ estimates of $e^{-tA_a}$, which play an important 
role in the proof of Theorem \ref{attainthm},  
were established by Kobayashi and Shibata \cite{kosh1998} 
in the three-dimensional case 
and further developed by Enomoto and Shibata \cite{ensh2004,ensh2005} 
for $n\geq 3$. We also note that 
$L^q$-$L^r$ estimates in the two-dimensional case were first 
established by Hishida \cite{hishida2016}, and 
recently Maekawa \cite{maekawa2019loc} derived those estimates
uniformly in small $a>0$ as a significant improvement of \cite{hishida2016}.

\begin{prop}[\cite{kosh1998,ensh2004,ensh2005}] \label{proplqlr}
\quad Let $n\geq 3$, $\sigma_0>0$ and assume $|a|\leq \sigma_0$.  
\begin{enumerate}
\item Let $1<q\leq r\leq\infty~(q\ne\infty)$. 
Then we have 
\begin{align}\label{lqlr}
&\|e^{-tA_a}f\|_{r}\leq Ct^{-\frac{n}{2}(\frac{1}{q}-\frac{1}{r})}\|f\|_q
\end{align}
for $t>0$ and $f\in L^{q}_\sigma(D)$, where $C=C(n,\sigma_0,q,r,D)>0$ 
is independent of $a$. 
\item Let $1<q\leq r\leq n$. Then we have 
\begin{align}\label{lqlrgrad}
&\|\nabla e^{-tA_a}f\|_{r}\leq Ct^{-\frac{n}{2}
(\frac{1}{q}-\frac{1}{r})-\frac{1}{2}}\|f\|_q
\end{align}
for $t>0$ and $f\in L^{q}_\sigma(D)$, where $C=C(n,\sigma_0,q,r,D)>0$ 
is independent of $a$. 
\item Let $n/(n-1)\leq q\leq r\leq \infty~(q\ne\infty)$. Then we have
\begin{align}\label{dual}
\|e^{-tA_a}P\nabla \cdot F\|_r\leq
Ct^{-\frac{n}{2}(\frac{1}{q}-\frac{1}{r})-\frac{1}{2}}\|F\|_q
\end{align}
for $t>0$ and $F\in L^q(D)$, where $C=C(n,\sigma_0,q,r,D)>0$ 
is independent of $a$. 
\end{enumerate}
\end{prop}
The proof of the assertion 3 is simply based on duality argument 
together with semigroup property especially for the case $r=\infty$. 
\par We also prepare the following lemma, which plays a role to 
prove the uniqueness within $Y$ defined by (\ref{yr}). 
\begin{lem}
Let $n\geq 3$ and $a>0$.  
For each precompact set $K\subset L^n_\sigma(D)$, we have 
\begin{align}\label{et0k}
\lim_{t\rightarrow 0}\sup_{f\in K}
t^{\frac{1}{2}}\big(\|e^{-t A_a}f\|_\infty+
\|\nabla e^{-t A_a}f\|_n\big)=0.
\end{align}
\end{lem}
\begin{proof}
\par By applying Proposition \ref{proplqlr} and 
approximating $f\in L^n_{\sigma}(D)$ 
by a sequence in $C^\infty_{0,\sigma}(D)$, 
we have  
\begin{align}\label{et0}
\lim_{t\rightarrow 0}t^{\frac{1}{2}}\big(\|e^{-t A_a}f\|_\infty+
\|\nabla e^{-t A_a}f\|_n\big)=0
\end{align}
for all $f\in L^n_{\sigma}(D)$. 
Given $\eta>0$, let $f_1,\cdots,f_m\in K$ fulfill 
$K\subset\displaystyle\bigcup_{j=1}^mB(f_j;\eta),$ where 
$B(f_j;\eta):=\{g\in  L^n_{\sigma}(D)\mid\|g-f_j\|_{n}<\eta\}$. 
For each $f\in K$, we choose $f_i\in K$ such that $f\in B(f_i;\eta)$. 
Then it follows from (\ref{lqlr}) that   
\begin{align*}
t^{\frac{1}{2}}\|e^{-tA_a}f\|_{\infty}&
\leq t^{\frac{1}{2}}\|e^{-tA_a}f_i\|_{\infty}
+t^{\frac{1}{2}}\|e^{-tA_a}(f-f_i)\|_{\infty}\\
& \leq t^{\frac{1}{2}}\|e^{-tA_a}f_i\|_{\infty}+C\|f-f_i\|_{n}
\leq\sum_{j=1}^mt^{\frac{1}{2}}
\|e^{-tA_a}f_j\|_{\infty}+C\eta.
\end{align*}
Since the right-hand side is independent of $f\in K$ and since  
$\eta$ is arbitrary, (\ref{et0}) 
yields 
\begin{align*}
\lim_{t\rightarrow 0}\sup_{f\in K}
t^{\frac{1}{2}}\|e^{-tA_a}f\|_{\infty}=0.
\end{align*}
We can discuss the $L^n$ norm of the first derivative similarly 
and thus conclude (\ref{et0k}). 
\end{proof}

We recall a function space $Y_{0}$  
defined by (\ref{yr0}), which is a 
Banach space equipped with 
norm $\|\cdot\|_{Y}=\|\cdot\|_{Y,\infty}$, where 
\begin{align*}
&\|v\|_{Y,t}:=[v]_{n,t}+[v]_{\infty,t}+[\nabla v]_{n,t},\\
&[v]_{q,t}:=
\sup_{0<\tau<t}\tau^{\frac{1}{2}-\frac{n}{2q}}\|v(\tau)\|_{q},\quad
q=n,\infty;\quad \quad  
[\nabla v]_{n,t}:=\sup_{0<\tau<t}\tau^{\frac{1}{2}}\|\nabla v(\tau)\|_{n}
\end{align*}
for $t\in (0,\infty]$. Construction of the solution is based on 
the following.
\begin{lem}\label{key}
Suppose $0<a^{(n-2)/(n+1)}<\delta$, where $\delta$ is a constant in 
Theorem \ref{rmkrho} with $(\ref{staclass2})$--$(\ref{rho2}).$ 
Let $\psi$ be a function on $\R$ satisfying $(\ref{psidef})$ 
and set $M=\displaystyle\max_{t\in\R}|\psi'(t)|$. 
Suppose that $u_s$ is the stationary solution obtained 
in Theorem \ref{rmkrho}.
For $u,v\in Y_{0}$, we set 
\begin{align*}
&G_1(u,v)(t)=\int_0^t e^{-(t-\tau)A_a}
P[u\cdot\nabla v](\tau)\,d\tau,\quad
G_2(v)(t)=\int_0^t 
e^{-(t-\tau)A_a}P[\psi(\tau)v\cdot\nabla u_s]\,d\tau,\\
&G_3(v)(t)=\int_0^t 
e^{-(t-\tau)A_a}P[\psi(\tau)u_s\cdot\nabla v]\,d\tau,\\
&G_4(v)(t)=\int_0^t 
e^{-(t-\tau)A_a}P\left[(1-\psi(\tau))a
\frac{\partial v}{\partial x_1}(\tau)\right]\,d\tau,\\
&H_1(t)=\int_0^t
e^{-(t-\tau)A_a}Ph_1(\tau)\,d\tau,\quad 
H_2(t)=\int_0^t
e^{-(t-\tau)A_a}Ph_2(\tau)\,d\tau,
\end{align*}
where $h_1$ and $h_2$ are defined 
by $(\ref{h1})$ and $(\ref{h2})$, respectively. 
Then we have $G_1(u,v),G_i(v),H_j$ $\in Y_{0}$ $(i=2,3,4,j=1,2)$ along with
\begin{align}
&\|G_1(u,v)\|_{Y,t}\leq C[u]^\frac{1}{2}_{n,t}[u]^\frac{1}{2}_{\infty,t}
[\nabla v]_{n,t},\label{g1est}\\ 
&\|G_2(v)\|_{Y,t}\leq C\big(\|\nabla u_s\|_{\frac{n}{2+\rho_3}}
+\|\nabla u_s\|_{\frac{n}{2}}+
\|\nabla u_s\|_{\frac{n}{2-\rho_4}}\big)[v]_{\infty,t},\label{g2est}\\
&\|G_3(v)\|_{Y,t}\leq C\big(\|u_s\|_{\frac{n}{1+\rho_1}}+\|u_s\|_{n}+
\|u_s\|_{\frac{n}{1-\rho_2}}\big)[\nabla v]_{n,t},\label{g3est}\\
&\|G_4(v)\|_{Y,t}\leq Ca[\nabla v]_{n,t},\label{g4est}\\
&\|H_1\|_{Y,t}\leq CM\|u_s\|_n,\label{h1est}\\
&\|H_2\|_{Y,t}\leq C\big(\|u_s\|_{\frac{n}{1-\rho_2}}
\|\nabla u_s\|_{\frac{n}{2-\rho_4}}+a
\|\nabla u_s\|_{\frac{n}{2-\rho_4}}\big)\label{h2est}
\end{align}
for all $t\in (0,\infty]$ and 
\begin{align}\label{h0}
\lim_{t\rightarrow 0}\|H_j(t)\|_{Y,t}=0
\end{align}
for $j=1,2$. 
Here, $C$ is a positive constant independent of $u,v,\psi,a$ and $t$.
\end{lem}

\begin{proof}
The continuity of those functions in $t$ is deduced by use of properties of 
analytic semigroups together with Proposition \ref{lqlr} in the same way 
as in Fujita and Kato \cite{fuka1964}. 
Since $L^\infty$ estimate is always the same as $L^n$ estimate of 
the first derivative, the estimate of $[\cdot]_{\infty,t}$ may be omitted. 
Although $(\ref{g1est})$--$(\ref{g3est})$ are discussed 
in Enomoto and Shibata \cite[Lemma 3.1.]{ensh2005} 
we briefly give the proof for completeness. 
We find that $u\in Y_0$ satisfies $u(t)\in L^{2n}(D)$ and  
\begin{align*}
\|u(t)\|_{2n}\leq 
t^{-\frac{1}{4}}[u]^{\frac{1}{2}}_{n,t}[u]^{\frac{1}{2}}_{\infty,t}
\end{align*}
for all $t>0$, which together with 
Proposition \ref{proplqlr} implies 
\begin{align*}
\int_0^t \|e^{-(t-\tau)A_a}P[u\cdot\nabla v](\tau)\|_n\,d\tau
\leq C\int_0^t(t-\tau)^{-\frac{1}{4}}\|u(\tau)\|_{2n}
\|\nabla v(\tau)\|_n\,d\tau
\leq C[u]^{\frac{1}{2}}_{n,t}[u]^{\frac{1}{2}}_{\infty,t}[\nabla v]_{n,t}
\end{align*}
and 
\begin{align*}
\int_0^t \|\nabla e^{-(t-\tau)A_a}P[u\cdot\nabla v](\tau)\|_n\,d\tau
&\leq C\int_0^t(t-\tau)^{-\frac{3}{4}}\|u(\tau)\|_{2n}
\|\nabla v(\tau)\|_n\,d\tau\\&\leq Ct^{-\frac{1}{2}}
[u]^{\frac{1}{2}}_{n,t}[u]^{\frac{1}{2}}_{\infty,t}[\nabla v]_{n,t}.
\end{align*}
We thus conclude (\ref{g1est}). 
It follows from Proposition \ref{proplqlr} that 
\begin{align}
&\int_0^t \|e^{-(t-\tau)A_a}P[\psi(\tau)v\cdot\nabla u_s]\|_n\,d\tau
\leq C\int_0^t(t-\tau)^{-\frac{1}{2}}\|v(\tau)\|_\infty
\|\nabla u_s\|_{\frac{n}{2}}\,d\tau
\leq C[v]_{\infty,t}\|\nabla u_s\|_{\frac{n}{2}}\label{g21}
\end{align}
and that 
\begin{align}\label{g23}
\int_0^t \|\nabla 
e^{-(t-\tau)A_a}P[\psi(\tau)v\cdot\nabla u_s]\|_n\,d\tau
&\leq C\int_0^t(t-\tau)^{-1+\frac{\rho_4}{2}}\|v(\tau)\|_\infty
\|\nabla u_s\|_{\frac{n}{2-\rho_4}}\,d\tau\nonumber
\\&\leq Ct^{-\frac{1}{2}+\frac{\rho_4}{2}}
[v]_{\infty,t}\|\nabla u_s\|_{\frac{n}{2-\rho_4}}
\end{align}
for $t>0$. Furthermore, for $t\geq 2$, we split the integral into 
\begin{align}
\int_0^t \|\nabla 
e^{-(t-\tau)A_a}P[\psi(\tau)v\cdot\nabla u_s]\|_n\,d\tau
=\int_0^{\frac{t}{2}}+\int_{\frac{t}{2}}^{t-1}+\int_{t-1}^t
\end{align}
as in \cite{chen1993} and \cite{ensh2005}. 
By applying (\ref{lqlrgrad}), we have 
\begin{align}
&\int_0^{\frac{t}{2}}\leq C\int_0^{\frac{t}{2}}(t-\tau)^{-1}
\|v(\tau)\|_\infty\|\nabla u_s\|_{\frac{n}{2}}\,d\tau\leq Ct^{-\frac{1}{2}}
[v]_{\infty,t}\|\nabla u_s\|_{\frac{n}{2}},\\
&\int_{\frac{t}{2}}^{t-1}\leq C\int_{\frac{t}{2}}^{t-1}
(t-\tau)^{-1-\frac{\rho_3}{2}}\|v(\tau)\|_\infty
\|\nabla u_s\|_{\frac{n}{2+\rho_3}}\,d\tau\leq 
Ct^{-\frac{1}{2}}[v]_{\infty,t}\|\nabla u_s\|_{\frac{n}{2+\rho_3}},\\
&\int_{t-1}^t\leq C\int_{t-1}^t(t-\tau)^{-1+\frac{\rho_4}{2}}
\|v(\tau)\|_\infty\|\nabla u_s\|_{\frac{n}{2-\rho_4}}\,d\tau\leq 
Ct^{-\frac{1}{2}}[v]_{\infty,t}\|\nabla u_s\|_{\frac{n}{2-\rho_4}}.
\label{g28}
\end{align}
Combining (\ref{g21})--(\ref{g28}) yields (\ref{g2est}). 
By the same manner, we obtain (\ref{g3est}). 
We use Proposition \ref{proplqlr} to find 
\begin{align*}
\int_0^t\left\|\nabla ^ke^{-(t-\tau)A_a}P\left[\big(1-\psi(\tau)\big)a
\frac{\partial v}{\partial x_1}\right]\right\|_nd\tau
&\leq Ca\int_0^{\min\{1,t\}}(t-\tau)^{-\frac{k}{2}}
\|\nabla v(\tau)\|_n\,d\tau\nonumber\\ 
&\leq Ca[\nabla v]_{n,t}
\int_0^{\min\{1,t\}}(t-\tau)^{-\frac{k}{2}}
\tau^{-\frac{1}{2}}\,d\tau
\end{align*}
for $k=0,1$, which lead us to (\ref{g4est}). 
We see (\ref{h1est}) from 
\begin{align}\label{h1est2}
\int_0^t\left\|\nabla^k e^{-(t-\tau)A_a}
P[\psi'(\tau)u_s]\right\|_nd\tau\leq CM\|u_s\|_{n}
\int_0^{\min\{1,t\}}(t-\tau)^{-\frac{k}{2}}\,d\tau
\end{align}
for $k=0,1$ and (\ref{h2est}) from 
\begin{align}\label{h2est2}
&\int_0^t\left\|\nabla^ke^{-(t-\tau)A_a}
P\left[\psi(\tau)(1-\psi(\tau))
\Big(u_s\cdot\nabla u_s
+a\frac{\partial u_s}{\partial x_1}\Big)\right]\right\|_n\,d\tau
\nonumber\\&\hspace{2cm}\leq 
C\|u_s\|_{\frac{n}{1-\rho_2}}\|\nabla u_s\|_{\frac{n}{2-\rho_4}}
\int_0^{\min\{1,t\}}(t-\tau)^{\frac{\rho_2+\rho_4}{2}-1-\frac{k}{2}}\,d\tau
\nonumber\\&\hspace{4cm}+Ca\|\nabla u_s\|_{\frac{n}{2-\rho_4}}
\int_0^{\min\{1,t\}}(t-\tau)^{-\frac{1}{2}+\frac{\rho_4}{2}-\frac{k}{2}}
\,d\tau
\end{align}
for $k=0,1$, where the condition (\ref{rho2}) is used. 
The behavior of $G_1(u,v)(t)$ and $G_i(v)(t)$ as well as the one of 
$H_j(t)$, see (\ref{h0}), as $t\rightarrow 0$ 
follows from (\ref{g1est})--(\ref{g4est}) 
and (\ref{h1est2})--(\ref{h2est2}) 
with $t<1$, so that $G_1(u,v),G_i(v),H_j\in Y_0$ 
and $\|G_1(u,v)(t)\|_n+\|G_i(v)(t)\|_n+\|H_j(t)\|_n\rightarrow 0$ 
as $t\rightarrow 0$. The proof is complete. 
\end{proof}

Let us construct a solution of (\ref{NS4}) by applying Lemma \ref{key}. 
\begin{prop}\label{yr0existence}
Let $\delta$ be the constant in 
Theorem \ref{rmkrho} with $(\ref{staclass2})$--\,$(\ref{rho2})$. 
Let $\psi$ be a function on $\R$ satisfying $(\ref{psidef})$ 
and set $M=\displaystyle\max_{t\in\R}|\psi'(t)|$.
Then there exists a constant $\varepsilon=\varepsilon(n,D)\in(0,\delta]$ such that 
if\, $0<(M+1)a^{(n-2)/(n+1)}<\varepsilon$, 
$(\ref{NS4})$ admits a solution $v\in Y_{0}$ with 
\begin{align}\label{apriori}
\|v\|_Y\leq C(M+1)a^\frac{n-2}{n+1}
\end{align}
and 
\begin{align}\label{vt0}
\lim_{t\rightarrow 0}\|v(t)\|_n=0.
\end{align}
\end{prop}

\begin{proof}
We set 
\begin{alignat}{3}\label{vm}
&v_{0}(t)=0,\nonumber\\ 
&v_{m+1}(t)=\int_0^t e^{-(t-\tau)A_a}P\Big[-v_m\cdot\nabla v_m
-\psi(\tau)v_m\cdot\nabla u_s&&-\psi(\tau)u_s\cdot\nabla v_m
+(1-\psi(\tau))a\frac{\partial v_m}{\partial x_1}\nonumber\\
&&&+h_1(\tau)+h_2(\tau)\Big]d\tau
\end{alignat}
for $m\geq 0.$ It follows from Theorem \ref{rmkrho}, Lemma \ref{key} 
and $a\in (0,1)$ that $v_m\in Y_0$ together with 
\begin{align}
&\|v_{m}\|_{Y,t}\leq \|G_1(v_{m-1},v_{m-1})\|_{Y,t}
+\sum_{i=2}^4\|G_i(v_{m-1})\|_{Y,t}
+\|H_1\|_{Y,t}+\|H_2\|_{Y,t},
\label{vmt}\\
&\|v_{m}\|_{Y}\leq C_1\|v_{m-1}\|_{Y}^2
+C_2a^\frac{n-2}{n+1}\|v_{m-1}\|_{Y}+C_3(M+1)a^\frac{n-2}{n+1},
\label{vmc3}\\
&\|v_{m+1}-v_m\|_{Y}\leq 
\{C_1(\|v_m\|_{Y}+\|v_{m-1}\|_{Y})+C_2a^\frac{n-2}{n+1}\}
\|v_m-v_{m-1}\|_{Y}\nonumber
\end{align}
for all $m\geq 1$. Hence, if we assume  
\begin{align}\label{assump}
(M+1)a^\frac{n-2}{n+1} <\min\left\{\delta,
\frac{1}{2C_2},\frac{1}{16C_1C_3}\right\}=:\varepsilon,
\end{align}
it holds that 
\begin{align}
&\|v_m\|_{Y}\leq 
\frac{1-C_2a^\frac{n-2}{n+1}-\sqrt{\big(1-C_2a^\frac{n-2}{n+1})^2
-4C_1C_3(M+1)a^\frac{n-2}{n+1}}}{2C_1}
\leq 4C_3(M+1)a^\frac{n-2}{n+1},\label{vmapriori}\\
&\|v_{m+1}-v_m\|_{Y}\leq 
\{8C_1C_3(M+1)a^\frac{n-2}{n+1}+C_2a^\frac{n-2}{n+1}\}
\|v_{m}-v_{m-1}\|_{Y}\nonumber
\end{align}
for all $m\geq 1$ and that 
\begin{align*}
8C_1C_3(M+1)a^\frac{n-2}{n+1}+C_2a^\frac{n-2}{n+1}<1.
\end{align*}
Therefore, 
we obtain a solution $v\in Y_{0}$ satisfying (\ref{apriori}) with $C=4C_3$.
Moreover, by letting $m\rightarrow\infty$ in (\ref{vmt}) and by 
using (\ref{g1est})--(\ref{g4est}) and (\ref{h0}), we have 
(\ref{vt0}), which completes the proof.
\end{proof}
\begin{rmk}\label{rmkexist}
Let $b\in L_{\sigma}^n(D)$. By the same procedure, 
we can also construct a solution $T(t)b:=v(t)\in Y_{0}$ 
for the integral equation  
\begin{align}\label{bint}
v(t)=e^{-tA_a}b+\int_0^t e^{-(t-\tau)A_a}P\Big[-v\cdot\nabla v
&-\psi(\tau)v\cdot\nabla u_s-\psi(\tau)u_s\cdot\nabla v\nonumber\\
&+(1-\psi(\tau))a\frac{\partial v}{\partial x_1}
+h_1(\tau)+h_2(\tau)\Big]d\tau
\end{align}
whenever
\begin{align*}
\|b\|_n+(M+1)a^\frac{n-2}{n+1}<\min\left\{\delta,
\frac{1}{2C_2},\frac{1}{16C_1C_0},\frac{1}{16C_1C_3}\right\}
\end{align*}
is satisfied. 
Here, the constant $C_0$ is determined by the following three estimates:
\begin{align*}
\|e^{-tA_a}b\|_q\leq C_0t^{-\frac{1}{2}+\frac{3}{2q}}\|b\|_n,
\quad\quad q=n,\infty;\quad\quad
\|\nabla e^{-tA_a}b\|_n\leq C_0t^{-\frac{1}{2}}\|b\|_n.
\end{align*}
Moreover, we find that the solution $T(t)b$ is estimated by 
\begin{align*}
\|T(\cdot)b\|_{Y}\leq 4\big(C_0\|b\|_n+C_3(M+1)a^\frac{n-2}{n+1}\big).
\end{align*}
This will be used in the proof of uniqueness of solutions within $Y$, see 
(\ref{yr}).  
\end{rmk}
\vspace{0.4cm}

We further derive sharp decay properties of the solution $v(t)$ 
obtained above. To this end, 
the first step is the following. 
In what follows, for simplicity of notation, we write 
\begin{align*}
G_1(t)=G_1(v,v)(t),\quad G_i(t)=G_i(v)(t)
\end{align*}
for $i=2,3,4$, which are defined in Lemma \ref{key}.
\begin{lem}\label{propsharpdecay}
Let $\varepsilon$ be the constant in  
Proposition \ref{yr0existence}. 
Given $\rho\in (0,1)$ satisfying $\rho\leq \min\{\rho_1,\rho_3\}$, 
there exists a constant 
$\varepsilon'=\varepsilon'(\rho,n,D)\in(0,\varepsilon]$ 
such that if\, $0<(M+1)a^{(n-2)/(n+1)}$$<\varepsilon'$, then 
the solution $v(t)$ obtained in Proposition \ref{yr0existence} 
satisfies
\begin{align}
&\|v(t)\|_q=O(t^{-\frac{1}{2}+\frac{n}{2q}-\frac{\rho}{2}}),
\quad\quad n\leq \forall q\leq\infty,
\label{sharpdecay2}\\
&\|\nabla v(t)\|_n=O(t^{-\frac{1}{2}-\frac{\rho}{2}})
\label{sharpdecaygrad2}
\end{align}
as $t\rightarrow\infty$.   
\end{lem}

\begin{proof}
We start with the case $q=n$, that is, 
\begin{align}\label{nsharp}
\|v(t)\|_n=O(t^{-\frac{\rho}{2}})
\end{align} 
as $t\rightarrow\infty.$
By using (\ref{lqlr}), we have 
\begin{align}
\|G_1(t)\|_n&\leq
Ct^{-\frac{\rho}{2}}
\big(\sup_{0<\tau<t}\tau^{\frac{1}{2}}\|\nabla v(\tau)\|_n\big)
\big(\sup_{0<\tau<t}\tau^{\frac{\rho}{2}}\|v(\tau)\|_n\big)
\leq Ct^{-\frac{\rho}{2}}\|v\|_Y
\sup_{0<\tau<t}\tau^{\frac{\rho}{2}}\|v(\tau)\|_n,
\label{g1rho1}\\
\|G_2(t)\|_n&\leq Ct^{-\frac{\rho_3}{2}}\|\nabla u_s\|_{\frac{n}{2+\rho_3}}
\sup_{0<\tau<t}\tau^{\frac{1}{2}}\|v(\tau)\|_\infty 
\leq Ct^{-\frac{\rho_3}{2}}
\|\nabla u_s\|_{\frac{n}{2+\rho_3}}\|v\|_Y\label{g2rho1}
\end{align} 
and 
\begin{align}\label{g3rho1}
\|G_3(t)\|_n\leq Ct^{-\frac{\rho_1}{2}}
\|u_s\|_{\frac{n}{1+\rho_1}}\|v\|_Y
\end{align}
for all $t>0$. Moreover, we obtain 
\begin{align}
\|G_4(t)\|_n\leq
 Ca\int_0^{\min\{1,t\}}(t-\tau)^{-\frac{1}{2}}\|v(\tau)\|_n\,
d\tau\leq Cat^{-\frac{1}{2}}\|v\|_Y\label{g4rho1}
\end{align}
for all $t>0$ by use of (\ref{dual}). 
From (\ref{lqlr}) we see that  
\begin{align}\label{h1rho1}
\|H_1(t)\|_n
\leq CM t^{-\frac{\rho_1}{2}}\|u_s\|_{\frac{n}{1+\rho_1}}
\end{align}
and that
\begin{align}\label{h2rho1}
\|H_2(t)\|_n\leq Ct^{-\frac{2-\rho_2}{2}}
\|u_s\|_{\frac{n}{1-\rho_2}}
\|\nabla u_s\|_{\frac{n}{2}}+Cat^{-\frac{1+\rho_3}{2}}
\|\nabla u_s\|_{\frac{n}{2+\rho_3}}
\end{align}
for $t>0$. Note that $\rho_2<1$, see (\ref{rho1}).
Collecting (\ref{g1rho1})--(\ref{h2rho1}) for $t>1$ and (\ref{apriori}) 
with $C=4C_3$ yields 
\begin{align*}
\sup_{0<\tau<t}\tau^{\frac{\rho}{2}}\|v(\tau)\|_n&\leq 
C_4\|v\|_Y
\sup_{0<\tau<t}\tau^{\frac{\rho}{2}}\|v(\tau)\|_n+C_5\\
&\leq 4C_3C_4(M+1)a^{\frac{n-2}{n+1}}
\sup_{0<\tau<t}\tau^{\frac{\rho}{2}}\|v(\tau)\|_n+C_5
\end{align*}
with some constants $C_4=C_4(\rho)>0$ and 
$C_5=C_5(\|v\|_Y,u_s,a,M,\rho_1,\rho_2,\rho_3)>0$
independent of $t$, where $C_3$ comes from estimates of $H_j(t)$ $(j=1,2)$ 
in (\ref{vmc3}).
Therefore, if we assume 
\begin{align*}
(M+1)a^{\frac{n-2}{n+1}}<\min\Big\{\varepsilon,\frac{1}{4C_3C_4}
\Big\}=:\varepsilon',
\end{align*}
we have $\|v(t)\|_n\leq Ct^{-\rho/2}$ for all $t>0$, 
which implies (\ref{nsharp}). 
\par We next show that
\begin{align*}
\|v(t)\|_\infty+\|\nabla v(t)\|_n=O(t^{-\frac{1}{2}-\frac{\rho}{2}})
\end{align*} 
as $t\rightarrow \infty$, 
which together with (\ref{nsharp}) implies 
$(\ref{sharpdecay2})$ and $(\ref{sharpdecaygrad2})$. 
It suffices to show that 
\begin{align}\label{inftygrad2}
t^{\frac{1}{2}}\|v(t)\|_\infty
+t^{\frac{1}{2}}\|\nabla v(t)\|_n\leq C\Big\|v\Big(\frac{t}{2}\Big)\Big\|_n
\end{align}
for all $t\geq 2$. The following argument is similar to Enomoto and 
Shibata \cite{ensh2005}.  
When $t\geq T>1$, we have
\begin{align}\label{Tt}
v(t)=e^{-(t-T)A_a}v(T)-\int_T^t e^{-(t-\tau)A_a}P\big[v\cdot\nabla v
+v\cdot\nabla u_s+u_s\cdot\nabla v\big]\,d\tau.
\end{align}
By the same argument as in the proof of Lemma \ref{key} and by 
(\ref{usest}), (\ref{assump}) as well as (\ref{apriori}) with $C=4C_3$, 
the integral of (\ref{Tt}) is estimated as 
\begin{align*}
\int_T^t&\|e^{-(t-\tau)A_a}P[\cdots]\|_\infty\,d\tau+
\int_T^t\|\nabla e^{-(t-\tau)A_a}P[\cdots]\|_n\,d\tau\\&\leq 
C_1(t-T)^{-\frac{1}{2}}\,
\big(\sup_{T\leq \tau\leq t}\|v(\tau)\|_n\big)^{\frac{1}{2}}\,
\big(\sup_{T\leq \tau\leq t}(\tau-T)^{\frac{1}{2}}\|v(\tau)\|_\infty
\big)^{\frac{1}{2}}\,
\big(\sup_{T\leq \tau\leq t}(\tau-T)^{\frac{1}{2}}\|\nabla v(\tau)\|_n
\big)\\&\quad\quad\quad\quad\quad\quad~+
C_2a^{\frac{n-2}{n+1}}(t-T)^{-\frac{1}{2}}
\big\{\sup_{T\leq \tau\leq t}(\tau-T)^{\frac{1}{2}}\|v(\tau)\|_\infty
+\sup_{T\leq \tau\leq t}(\tau-T)^{\frac{1}{2}}
\|\nabla v(\tau)\|_n\big\}\\
&\leq C_1(t-T)^{-\frac{1}{2}}\|v\|_Y
\sup_{T\leq \tau\leq t}(\tau-T)^{\frac{1}{2}}
\|\nabla v(\tau)\|_n\\
&\quad\quad\quad\quad\quad\quad~
+\frac{1}{2}(t-T)^{-\frac{1}{2}}
\big\{\sup_{T\leq \tau\leq t}(\tau-T)^{\frac{1}{2}}\|v(\tau)\|_\infty
+\sup_{T\leq \tau\leq t}(\tau-T)^{\frac{1}{2}}
\|\nabla v(\tau)\|_n\big\}\\
&\leq \frac{3}{4}(t-T)^{-\frac{1}{2}}
\sup_{T\leq \tau\leq t}(\tau-T)^{\frac{1}{2}}\|\nabla v(\tau)\|_n
+\frac{1}{2}(t-T)^{-\frac{1}{2}}
\sup_{T\leq \tau\leq t}(\tau-T)^{\frac{1}{2}}\|v(\tau)\|_\infty.
\end{align*}
Therefore, we have
\begin{align*}
\sup_{T\leq \tau\leq t}(\tau-T)^{\frac{1}{2}}\|\nabla v(\tau)\|_n
+\sup_{T\leq \tau\leq t}(\tau-T)^{\frac{1}{2}}\|v(\tau)\|_\infty\leq 
C\|v(T)\|_n
\end{align*}
for all $t\geq T$. This combined with 
$t^{1/2}\leq \sqrt{2}(t-T)^{1/2}$ for $t\geq 2T$ asserts that
\begin{align*}
t^{\frac{1}{2}}\|\nabla v(t)\|_n
+t^{\frac{1}{2}}\|v(t)\|_\infty\leq C\|v(T)\|_n
\end{align*}
for all $t\geq 2T$. We then put $T=t/2~(t\geq 2)$ to conclude
(\ref{inftygrad2}).
\end{proof}

Sharp decay properties (\ref{sharpdecay})--(\ref{sharpdecaygrad}) 
for the case $n=3$ are established in the following proposition. 
\begin{prop}\label{propsharp3}
Let $n=3$ and set $\varepsilon_*:=\varepsilon'(\rho,3,D)$ which is the 
constant in Lemma \ref{propsharpdecay} with $\rho:=\min\{\rho_1,\rho_3\}$ 
$($recall that $0<\rho_1<1/2$, $0<\rho_3<1/4$ for $n=3$$).$
If \,$0<(M+1)a^{1/4}<\varepsilon_*$, then 
the solution $v(t)$ obtained in Proposition \ref{yr0existence} 
enjoys $(\ref{sharpdecay})$ and $(\ref{sharpdecaygrad}).$
\end{prop}

\begin{proof}
The case $\rho_1\leq\rho_3$ directly follows from 
Lemma \ref{propsharpdecay}. 
To discuss the other case $\rho_3<\rho_1$, we show by induction 
that if $0<(M+1)a^{1/4}<\varepsilon_*$, then 
\begin{align}\label{sigmak3}
\|v(t)\|_3=O(t^{-\sigma_k}),\qquad\sigma_k:
=\min\Big\{\frac{k}{2}\rho_3,\,\frac{\rho_1}{2}\Big\}
\end{align}
as $t\rightarrow\infty$ for all $k\geq 1$. 
We already know (\ref{sigmak3}) with $k=1$ from Lemma \ref{propsharpdecay}. 
\par Let $k\geq 2$ and suppose (\ref{sigmak3}) with $k-1$. By taking 
(\ref{apriori}) (near $t=0$) and (\ref{inftygrad2}) into account, we have 
\begin{align*}
J_{k-1}(v):=\sup_{\tau>0}(1+\tau)^{\sigma_{k-1}}
\|v(\tau)\|_3+
\sup_{\tau>0}\tau^{\frac{1}{2}}(1+\tau)^{\sigma_{k-1}}
\big(\|v(\tau)\|_\infty+\|\nabla v(\tau)\|_3\big)<\infty.
\end{align*}
We use this to see that
\begin{align*}
\|G_1(t)\|_3&\leq C\int_0^t(t-\tau)^{-\frac{1}{2}}\tau^{-\frac{1}{2}}
(1+\tau)^{-2\sigma_{k-1}}\,d\tau\\
&\qquad\qquad
\times\big(\sup_{\tau>0}(1+\tau)^{\sigma_{k-1}}
\|v(\tau)\|_3\big)
\big(\sup_{\tau>0}\tau^{\frac{1}{2}}(1+\tau)^{\sigma_{k-1}}
\|\nabla v(\tau)\|_3\big)\\
&\leq Ct^{-2\sigma_{k-1}}J_{k-1}(v)^2,
\end{align*}
and that
\begin{align*}
\|G_2(t)\|_3&\leq C\int_0^t(t-\tau)^{-\frac{1+\rho_3}{2}}
\tau^{-\frac{1}{2}}(1+\tau)^{-\sigma_{k-1}}\,d\tau\,
\|\nabla u_s\|_{\frac{3}{2+\rho_3}}
\sup_{\tau>0}\tau^{\frac{1}{2}}(1+\tau)^{\sigma_{k-1}}\|v(\tau)\|_\infty
\\
&\leq Ct^{-\frac{\rho_3}{2}-\sigma_{k-1}}
\|\nabla u_s\|_{\frac{3}{2+\rho_3}}J_{k-1}(v)
\end{align*}
for $t>0$ due to $\sigma_{k-1}\leq\rho_1/2<1/4$. 
From these and (\ref{g3rho1})--(\ref{h2rho1}), we obtain 
(\ref{sigmak3}) with $k$. 
We thus conclude (\ref{sharpdecay}) with $q=3$, 
which together with (\ref{inftygrad2}) completes the proof.
\end{proof}

To derive even more rapid decay properties of the solution $v(t)$ 
for $n\geq 4$, 
we need the following lemma, 
which gives the $L^{q_0}$-decay of $v(t)$
with a specific $q_0$, see (\ref{q0}) below. 
\begin{lem}\label{lemsharp}
Let $n\geq 4$. Suppose $1<\rho_1\leq 1+\rho_3$ in addition to 
$(\ref{rho1})$ $($the set of those parameters is 
nonvoid when $n\geq 4)$. 
Let $\varepsilon$ be the constant in Proposition \ref{yr0existence} and 
$v(t)$ the solution obtained there. 
Given $\gamma$ satisfying
\begin{align}\label{gamma} 
\max\Big\{0,\,\frac{\rho_1+3-n}{2}\Big\}<\gamma<\frac{1}{2}
\end{align}
$($note that $(\ref{rho1})$ yields $\rho_1<n-2)$, there exists a constant 
$\varepsilon''=\varepsilon''(\gamma,n,D)
\in (0,\varepsilon]$ such that if\, $0<(M+1)a^{(n-2)/(n+1)}
<\varepsilon''$, 
then $v(t)\in L^{q_0}(D)$ for all $t>0$ and 
\begin{align}\label{vq}
\sup_{\tau>0}(1+\tau)^\gamma\|v(\tau)\|_{q_0}<\infty,
\end{align}
where 
\begin{align}\label{q0}
q_0:=\frac{n}{1+\rho_1-2\gamma}\,(<n).
\end{align}
\end{lem}

\begin{proof}
We show that there exists a constant 
$\varepsilon''(\gamma,n,D)\in (0,\varepsilon]$ such that if 
$0<(M+1)a^{(n-2)/(n+1)}<\varepsilon''$, then 
$v_m(t)\in L^{q_0}(D)$ for all $t>0$ along with 
\begin{align}\label{vm2}
&K_m:=\sup_{\tau>0}(1+\tau)^\gamma\|v_m(\tau)\|_{q_0}<\infty,
&K_m\leq \frac{1}{2}K_{m-1}+C(M+1)a^{\frac{n-1}{n+1}}
\end{align}
for all $m\geq 1$,
where $v_m(t)$ is the approximate 
solution defined by (\ref{vm}) and 
$C$ is a positive constant independent of $a$ and $m$. We use 
(\ref{lqlr}) to see that
\begin{align}\label{h1q0}
\int_0^t\|e^{-(t-\tau)A_a}Ph_1(\tau)\|_{q_0}\,d\tau
\leq CM\|u_s\|_{\frac{n}{1+\rho_1}}
\int_0^{\min\{1,t\}} (t-\tau)^{-\gamma}\,d\tau
\leq CM\|u_s\|_{\frac{n}{1+\rho_1}}(1+t)^{-\gamma}
\end{align}
for $t>0$. Moreover, it holds that 
\begin{align*}
\int_0^t\Big\|e^{-(t-\tau)A_a}
P\Big[\psi(\tau)\big(1-\psi(\tau)\big)a
\frac{\partial u_s}{\partial x_1}\Big]\Big\|_{q_0}\,d\tau&
\leq Ca\|\nabla u_s\|_{r}
\end{align*}
for $t\leq 2$, where $r:=\min\{n/(2-\rho_4),q_0\}$ and that
\begin{align*}
\int_0^t\Big\|e^{-(t-\tau)A_a}
P\Big[\psi(\tau)\big(1-\psi(\tau)\big)a
\frac{\partial u_s}{\partial x_1}\Big]\Big\|_{q_0}\,d\tau
&\leq Ca\|\nabla u_s\|_{\frac{n}{2+\rho_3}}
\int_0^1(t-\tau)^{-\gamma-\frac{1+\rho_3-\rho_1}{2}}\,d\tau\\
&\leq Ca\|\nabla u_s\|_{\frac{n}{2+\rho_3}}t^{-\gamma}
\end{align*}
for $t>2$ as well as that 
\begin{align*}
\int_0^t\|e^{-(t-\tau)A_a}
P[\psi(\tau)\big(1-\psi(\tau)\big)u_s\cdot\nabla u_s]\|_{q_0}\,d\tau&\leq 
C\|u_s\|_{\frac{n}{1+\kappa}}
\|\nabla u_s\|_{\frac{n}{2}}\int_0^{\min\{1,t\}}
(t-\tau)^{-1-\gamma+\frac{\rho_1-\kappa}{2}}\,d\tau\\
&\leq C\|u_s\|_{\frac{n}{1+\kappa}}
\|\nabla u_s\|_{\frac{n}{2}}(1+t)^{-\gamma}
\end{align*}
for $t>0$, where 
$\max\{0,\rho_1-2\}<\kappa<\min\{n-3,\rho_1-2\gamma\}$ 
(note that (\ref{rho1}) yields $\rho_1<n-1$). These estimates imply
\begin{align*}
\int_0^t\|e^{-(t-\tau)A_a}Ph_2(\tau)\|_{q_0}\,d\tau
\leq C(a\|\nabla u_s\|_{r}+a\|\nabla u_s\|_{\frac{n}{2+\rho_3}}+
\|u_s\|_{\frac{n}{1+\kappa}}
\|\nabla u_s\|_{\frac{n}{2}})(1+t)^{-\gamma}
\end{align*}
for $t>0$, which together with (\ref{h1q0}) and (\ref{usest})
leads us to $v_1(t)\in L^{q_0}(D)$ for all $t>0$ with 
\begin{align}\label{v1}
K_1\leq C(M+1)a^{\frac{n-1}{n+1}}.
\end{align}
This proves (\ref{vm2}) with $m=1$ since $K_0=0$. 
\par Let $m\geq 2$ and suppose that  
$v_{m-1}(t)\in L^{q_0}(D)$ for all $t>0$ and 
(\ref{vm2}) with $m-1$. Then we have 
$G_1(v_{m-1},v_{m-1})(t)\in L^{q_0}(D)$ for $t>0$ with 
\begin{align}\label{g1m}
\sup_{\tau>0}(1+\tau)^{\gamma}\|G_1(v_{m-1},v_{m-1})(\tau)\|_{q_0}
\leq CK_{m-1}
\sup_{\tau>0}\tau^{\frac{1}{2}}\|\nabla v_{m-1}(\tau)\|_n.
\end{align}
Let $t\geq 2$ and split the integral into 
\begin{align*}
\int_0^t\|e^{-(t-\tau)A_a}P[\psi(\tau)u_s\cdot \nabla v_{m-1}]
\|_{q_0}\,d\tau
=\int_0^{\frac{t}{2}}+\int_{\frac{t}{2}}^{t-1}+\int_{t-1}^t. 
\end{align*}
Let $\lambda\in (0,\rho_1]$ satisfy $\lambda<n-3+2\gamma-\rho_1$; in fact, 
we can take such $\lambda$ due to (\ref{gamma}). 
Then (\ref{dual}) with $F=v_{m-1}\otimes u_s$ implies 
\begin{align*}
\int_0^{\frac{t}{2}}&\leq C\int_0^{\frac{t}{2}}(t-\tau)^{-1}\|u_s\|_n
\|v_{m-1}(\tau)\|_{q_0}\,d\tau\leq
Ct^{-\gamma}\|u_s\|_{n}K_{m-1},\\
\int_{\frac{t}{2}}^{t-1}&\leq C\int_{\frac{t}{2}}^{t-1}
(t-\tau)^{-1-\frac{\lambda}{2}}\|u_s\|_{\frac{n}{1+\lambda}}
\|v_{m-1}(\tau)\|_{q_0}\,d\tau\leq 
Ct^{-\gamma}\|u_s\|_{\frac{n}{1+\lambda}}K_{m-1},\\
\int_{t-1}^t&\leq C\int_{t-1}^t(t-\tau)^{-1+\frac{\rho_2}{2}}
\|u_s\|_{\frac{n}{1-\rho_2}}
\|v_{m-1}(\tau)\|_{q_0}\,d\tau\leq 
Ct^{-\gamma}\|u_s\|_{\frac{n}{1-\rho_2}}K_{m-1}
\end{align*}
for $t\geq 2$. Moreover, we use (\ref{dual}) again to see that 
\begin{align*}
\int_0^t\|e^{-(t-\tau)A_a}P[\psi(\tau)u_s\cdot \nabla v_{m-1}]
\|_{q_0}\,d\tau&\leq C\int_0^t(t-\tau)^{-1+\frac{\rho_2}{2}}
\|u_s\|_{\frac{n}{1-\rho_2}}\|v_{m-1}(\tau)\|_{q_0}\,d\tau\\
&\leq C\|u_s\|_{\frac{n}{1-\rho_2}}K_{m-1}
\end{align*}
for $t\leq 2$. We thus conclude 
$G_3(v_{m-1})(t)\in L^{q_0}(D)$ for $t>0$ with 
\begin{align}\label{g3m}
\sup_{\tau>0}(1+\tau)^\gamma\|G_3(v_{m-1})(\tau)\|_{q_0}\leq
C(\|u_s\|_n+\|u_s\|_{\frac{n}{1+\lambda}}+\|u_s\|_{\frac{n}{1-\rho_2}})
K_{m-1}.
\end{align}
By the same calculation, we have 
$G_2(v_{m-1})(t)\in L^{q_0}(D)$ for $t>0$ with 
\begin{align}\label{g2m}
\sup_{\tau>0}(1+\tau)^\gamma\|G_2(v_{m-1})(\tau)\|_{q_0}
\leq C(\|u_s\|_n+\|u_s\|_{\frac{n}{1+\lambda}}+\|u_s\|_{\frac{n}{1-\rho_2}})
K_{m-1}.
\end{align}
We also have 
\begin{align*}
\int_0^t\Big\|e^{-(t-\tau)A_a}P\Big[\big(1-\psi(\tau)\big)a
\frac{\partial v_{m-1}}{\partial x_1}\Big]\Big\|_{q_0}&\leq 
Ca\int_0^{\min\{1,t\}}(t-\tau)^{-\frac{1}{2}}\|v_{m-1}(\tau)\|_{q_0}\,d\tau
\\&\leq CaK_{m-1}(1+t)^{-\frac{1}{2}}\leq CaK_{m-1}(1+t)^{-\gamma}
\end{align*}
for $t>0$ by (\ref{dual}). 
This together with (\ref{v1})--(\ref{g2m}), 
(\ref{usest}) and (\ref{vmapriori}) yields 
$v_{m}(t)\in L^{q_0}(D)$ for $t>0$ and 
\begin{align*}
K_m&\leq C(M+1)a^{\frac{n-1}{n+1}}+
\widetilde{C}_1\Big\{
\big(\sup_{\tau>0}\tau^{\frac{1}{2}}\|\nabla v_{m-1}(\tau)\|_n\big)+
\|u_s\|_n+\|u_s\|_{\frac{n}{1+\lambda}}+\|u_s\|_{\frac{n}{1-\rho_2}}
+a\Big\}K_{m-1}\\
&\leq C(M+1)a^{\frac{n-1}{n+1}}+\widetilde{C}_1
(4C_3+\widetilde{C}_2)(M+1)a^{\frac{n-2}{n+1}}K_{m-1}.
\end{align*}
\par Suppose 
\begin{align*}
(M+1)a^{\frac{n-2}{n+1}}<\min\left\{\varepsilon,\,
\frac{1}{2\widetilde{C}_1(4C_3+\widetilde{C}_2)}\right\}=:\varepsilon'',
\end{align*}
then we get (\ref{vm2}) with $m$ and, thereby, conclude 
\begin{align*}
K_m\leq 2C(M+1)a^{\frac{n-1}{n+1}}
\end{align*}
for all $m\geq 1$. Since we know that 
$\|v_m(t)-v(t)\|_n\rightarrow 0$ as $m\rightarrow\infty$ for each $t>0$,
we obtain $v(t)\in L^{q_0}(D)$ for $t>0$ with 
\begin{align*}
\sup_{\tau>0}(1+\tau)^\gamma\|v(\tau)\|_{q_0}
\leq 2C(M+1)a^{\frac{n-1}{n+1}}<\infty,
\end{align*}
which completes the proof.
\end{proof}

In view of Lemma \ref{propsharpdecay} and Lemma \ref{lemsharp}, 
we prove sharp decay properties 
(\ref{sharpdecayn})--(\ref{sharpdecaygradn}) for $n\geq 4$.
\begin{prop}\label{propsharp}
Let $n\geq 4$. Suppose $\rho_3>1$ and $1<\rho_1\leq 1+\rho_3$ 
in addition to $(\ref{rho1})$ 
$($the set of those parameters is nonvoid when $n\geq 4).$ 
Let $\varepsilon$ be the constant in Proposition \ref{yr0existence}. 
There exists a constant
$\varepsilon_*=\varepsilon_*(n,D)\in(0,\varepsilon]$ 
such that if $0<(M+1)a^{(n-2)/(n+1)}<\varepsilon_*$, 
then the solution $v(t)$ obtained in Proposition \ref{yr0existence} enjoys 
$(\ref{sharpdecayn})$ and $(\ref{sharpdecaygradn})$.
\end{prop}

\begin{proof}
Fix $1/2<\rho<1$ and $\gamma>0$ such that 
\begin{align}\label{gamma2}
\max\Big\{\frac{1}{2}-\frac{\rho}{2},\,\frac{\rho_1+3-n}{2}\Big\}<\gamma<
\frac{1}{2}.
\end{align}
Let $\varepsilon'(\rho,n,D)$ and $\varepsilon''(\gamma,n,D)$ be 
the constants in 
Lemma \ref{propsharpdecay} and Lemma \ref{lemsharp}, respectively. 
We show by induction that if 
\begin{align*}
(M+1)a^{\frac{n-2}{n+1}}<
\min\{\varepsilon'(\rho,n,D),\varepsilon''(\gamma,n,D)\}
=:\varepsilon_*(n,D),
\end{align*}
then $v(t)$ satisfies 
\begin{align}\label{sigmak}
\|v(t)\|_n=O(t^{-\sigma_k}),\qquad \sigma_k:=\min\Big\{
\frac{k}{2}\rho,\,\frac{\rho_1}{2}\Big\}
\end{align}
as $t\rightarrow\infty$ for all $k\geq 1$. 
This implies (\ref{sharpdecayn}) with $q=n$, 
which together with (\ref{inftygrad2}) completes the proof. 
Since $\rho<\rho_1$, (\ref{sigmak}) with $k=1$ follows from 
Lemma \ref{propsharpdecay}. 
We note that $\sigma_1<1/2$ and $\sigma_k>1/2$ for $k\geq 2$.
\par Let $k\geq 2$ and suppose (\ref{sigmak}) with $k-1$. Then 
\begin{align*}
L_{k-1}(v):=\sup_{\tau>0}(1+\tau)^{\sigma_{k-1}}\|v(\tau)\|_n+
\sup_{\tau>0}\tau^{\frac{1}{2}}(1+\tau)^{\sigma_{k-1}}
\big(\|v(\tau)\|_\infty+\|\nabla v(\tau)\|_n\big)
<\infty
\end{align*}
holds due to (\ref{apriori}) (near $t=0$) as well as (\ref{inftygrad2}). 
In what follows, we always assume $t\geq 2$. 
From (\ref{vq}), it follows that 
\begin{align}
\|G_1(t)\|_n&\leq \int_0^{\frac{t}{2}}(t-\tau)^{-\frac{n}{2q_0}}
\|v(\tau)\|_{q_0}\|\nabla v(\tau)\|_n\,d\tau
+\int_{\frac{t}{2}}^t(t-\tau)^{-\frac{1}{2}}
\|v(\tau)\|_n\|\nabla v(\tau)\|_n\,d\tau=:I+II\label{g1n}
\end{align}
with 
\begin{align}
I\leq Ct^{-\frac{n}{2q_0}}
\big(\sup_{\tau>0}(1+\tau)^\gamma\|v(\tau)\|_{q_0}\big)L_{k-1}(v)
\leq Ct^{-\frac{\rho_1}{2}}
\big(\sup_{\tau>0}(1+\tau)^\gamma\|v(\tau)\|_{q_0}\big)
L_{k-1}(v),\label{I}
\end{align}
where (\ref{q0}) and (\ref{gamma2}) are taken into account and 
\begin{align}
II\leq Ct^{-2\sigma_{k-1}}L_{k-1}(v)^2.\label{II2}
\end{align}
For $G_2(t)$, we split the integral into 
\begin{align*} 
\int_0^t\|e^{-(t-\tau)A_a}
P[\psi(\tau)v\cdot\nabla u_s]\|_n\,d\tau
=\int_{0}^{\frac{t}{2}}+\int_{\frac{t}{2}}^{t-1}+
\int_{t-1}^{t}.
\end{align*}
Then we find 
\begin{align*}
\int_0^{\frac{t}{2}}&\leq 
C\int_0^{\frac{t}{2}}(t-\tau)^{-\frac{1+\rho_3}{2}}\tau^{-\frac{1}{2}}
(1+\tau)^{-\sigma_{k-1}}\,d\tau\,\|\nabla u_s\|_{\frac{n}{2+\rho_3}}
\big(\sup_{\tau>0}\tau^{\frac{1}{2}}
(1+\tau)^{\sigma_{k-1}}\|v(\tau)\|_\infty\big)\nonumber\\
&\leq 
\begin{cases}
Ct^{-\frac{\rho_3}{2}-\sigma_{k-1}}\|\nabla u_s\|_{\frac{n}{2+\rho_3}}
L_{k-1}(v)
\leq Ct^{-\sigma_{k}}\|\nabla u_s\|_{\frac{n}{2+\rho_3}}
L_{k-1}(v)
\quad \mbox{if}~k=2,\\
Ct^{-\frac{1+\rho_3}{2}}\|\nabla u_s\|_{\frac{n}{2+\rho_3}}
L_{k-1}(v)\leq 
Ct^{-\frac{\rho_1}{2}}\|\nabla u_s\|_{\frac{n}{2+\rho_3}}
L_{k-1}(v)
\quad \mbox{if}~k\geq 3
\end{cases}
\end{align*}
and 
\begin{align*}
\int_{\frac{t}{2}}^{t-1}+\int_{t-1}^{t}
\leq C
t^{-\sigma_{k-1}-\frac{1}{2}}\big(\|\nabla u_s\|_{\frac{n}{2+\rho_3}}
+\|\nabla u_s\|_{\frac{n}{2}}\big)L_{k-1}(v),
\end{align*}
where we have used $\rho_3>1$ and $\rho_1\leq 1+\rho_3.$
Estimates above imply that 
\begin{align}\label{g2n}
\|G_2(t)\|_n\leq Ct^{-\sigma_k}
\big(\|\nabla u_s\|_{\frac{n}{2+\rho_3}}
+\|\nabla u_s\|_{\frac{n}{2}}\big)L_{k-1}(v)
\end{align}
Similarly, we observe
\begin{align}
\|G_3(t)\|_n\leq Ct^{-\sigma_k}
\big(\|u_s\|_{\frac{n}{1+\rho_1}}
+\|u_s\|_{n}\big)L_{k-1}(v).
\end{align}
Moreover, by the same manner as 
in the proof of Lemma \ref{propsharpdecay}, we obtain 
\begin{align}
\|G_4(t)\|_n&\leq
Ct^{-\frac{n}{2q_0}}\sup_{\tau>0}(1+\tau)^\gamma\|v(\tau)\|_{q_0}
\leq Ct^{-\frac{\rho_1}{2}}\sup_{\tau>0}(1+\tau)^\gamma\|v(\tau)\|_{q_0},
\label{g4n}\\
\|H_1(t)\|_n&\leq CMt^{-\frac{\rho_1}{2}}\|u_s\|_{\frac{n}{1+\rho_1}},\\
\|H_2(t)\|_n&\leq Ct^{-\frac{2+\kappa}{2}}
\|u_s\|_{\frac{n}{1+\kappa}}\|\nabla u_s\|_{\frac{n}{2}}
+Ct^{-\frac{1+\rho_3}{2}}
\|\nabla u_s\|_{\frac{n}{2+\rho_3}}\nonumber\\
&\leq Ct^{-\frac{\rho_1}{2}}
\big(\|u_s\|_{\frac{n}{1+\kappa}}\|\nabla u_s\|_{\frac{n}{2}}
+\|\nabla u_s\|_{\frac{n}{2+\rho_3}}\big)\label{h2n}
\end{align}
for all $t\geq 2$, where $\kappa$ is chosen such that 
$\max\{0,\rho_1-2\}<\kappa<\min\{n-3,\rho_1\}$. 
Collecting (\ref{g1n})--(\ref{h2n}),
we conclude (\ref{sigmak}) with $k$.  
The proof is complete.
\end{proof}

We next consider the uniqueness. 
We begin with the classical result on the uniqueness of solutions within 
$Y_0$ as in Fujita and Kato \cite{fuka1964}. 
\begin{lem}\label{yr0uniqueness}
Let $\psi$ be a function on $\R$ satisfying $(\ref{psidef})$ and let  
$\delta$ be the constant in 
Theorem \ref{rmkrho} with $(\ref{staclass2})$--\,$(\ref{rho2}).$
Then there exists a constant 
$\tilde{\varepsilon}=\tilde{\varepsilon}(n,D)\in(0,\delta]$ 
such that if \,$0<a^{(n-2)/(n+1)}<\tilde{\varepsilon}$, 
$(\ref{NS4})$ admits at most one solution within $Y_{0}$. 
\end{lem}

\begin{proof}
The following argument is based on \cite{fuka1964}. 
Suppose that $v,\tilde{v}\in Y_{0}$ are solutions. 
Then we have
\begin{align}\label{v1v2t1}
\|v-\tilde{v}\|_{Y,t}\leq
\big\{C_1\big([\nabla v]_{n,t}+[\tilde{v}]_{n,t}^{\frac{1}{2}}
[\tilde{v}]^{\frac{1}{2}}_{\infty,t}\big)
+C_2a^{\frac{n-2}{n+1}}\big\}\|v-\tilde{v}\|_{Y,t},\quad t>0
\end{align} 
by applying (\ref{usest}) and Lemma \ref{key}. If we assume 
\begin{align}\label{assump2}
a^{\frac{n-2}{n+1}}<\min\Big\{\delta,\frac{1}{2C_2}\Big\}
=:\tilde{\varepsilon}
\end{align}
and choose $t_0>0$ such that 
\begin{align*}
C_1\big\{[\nabla v]_{n,t_0}+
\big(\sup_{0<\tau<\infty}\|\tilde{v}(\tau)\|_n\big)^{\frac{1}{2}}\,
[\tilde{v}]^{\frac{1}{2}}_{\infty,t_0}\big\}<\frac{1}{2},
\end{align*}
then (\ref{v1v2t1}) yields $[v-\tilde{v}]_{Y,t_0}=0$. 
Hence, we conclude $v=\tilde{v}$ on $(0,t_0]$ and obtain
\begin{align*}
v(t)-\tilde{v}(t)=\int_{t_0}^t
e^{-(t-\tau)A_a}P\Big[-(v&-\tilde{v})\cdot\nabla v
-\tilde{v}\cdot\nabla (v-\tilde{v})-\psi(\tau)(v-\tilde{v})\cdot\nabla u_s
\nonumber\\&
-\psi(\tau)u_s\cdot\nabla (v-\tilde{v})
+(1-\psi(\tau))a\frac{\partial}{\partial x_1}(v-\tilde{v})\Big]\,d\tau.
\end{align*} 
By the same argument as in the proof of Lemma \ref{key} 
together with (\ref{usest}), we see that
\begin{align}\label{t0test}
\|v-\tilde{v}\|_{Y,t_0,t}\leq C_*\|v-\tilde{v}\|_{Y,t_0,t}
\end{align}
for all $t>t_0$, where
\begin{align}
\|v\|_{Y,t_0,t}:&=
\sup_{t_0\leq\tau\leq t}\|v(\tau)\|_n
+\sup_{t_0\leq\tau\leq t}\|v(\tau)\|_\infty+
\sup_{t_0\leq \tau\leq t}\|\nabla v(\tau)\|_n,\\
C_*=C\Big[\big(t_0^{-\frac{1}{2}}\|v\|_Y&+t_0^{-\frac{1}{4}}
\|\tilde{v}\|_Y\big)
\big\{(t-t_0)^{\frac{3}{4}}
+(t-t_0)^{\frac{1}{4}}\big\}\nonumber\\&+a^\frac{n-1}{n+1}
\big\{(t-t_0)^{\frac{1}{2}}
+(t-t_0)^{\frac{\rho_2}{2}}+(t-t_0)^{\frac{\rho_4}{2}}\big\}
+a\big\{(t-t_0)+(t-t_0)^{\frac{1}{2}}\big\}
\Big]\label{c*}
\end{align}
and the constant $C$ is independent of $v$, $\tilde{v}$, $t$ and $t_0$. 
We choose $\eta>0$ such that 
\begin{align*}
\xi:=C\Big[\big(t_0^{-\frac{1}{2}}\|v\|_Y+t_0^{-\frac{1}{4}}
\|\tilde{v}\|_Y\big)
\big(\eta^{\frac{3}{4}}
+\eta^{\frac{1}{4}}\big)+a^\frac{n-1}{n+1}
\big(\eta^{\frac{1}{2}}
+\eta^{\frac{\rho_2}{2}}+\eta^{\frac{\rho_4}{2}}
\big)+a\big(\eta+\eta^{\frac{1}{2}}\big)
\Big]<1.
\end{align*}
On account of (\ref{t0test}), we have
$\|v-\tilde{v}\|_{Y,t_0,t_0+\eta}
\leq \xi\|v-\tilde{v}\|_{Y,t_0,t_0+\eta}$, 
which leads us to $v=\tilde{v}$ on $[t_0,t_0+\eta]$. 
By the same calculation, 
we can obtain (\ref{t0test})--(\ref{c*}), in 
which $t_0$ should be replaced by $t_0+\eta$ and hence
\begin{align*}
\|v-\tilde{v}\|_{Y,t_0+\eta,t_0+2\eta}&\leq 
C\Big[\big\{(t_0+\eta)^{-\frac{1}{2}}\|v\|_Y+(t_0+\eta)^{-\frac{1}{4}}
\|\tilde{v}\|_Y\big\}
\big(\eta^{\frac{3}{4}}
+\eta^{\frac{1}{4}}\big)\\
&\hspace{1.8cm}+a^\frac{n-1}{n+1}
\big(\eta^{\frac{1}{2}}
+\eta^{\frac{\rho_2}{2}}+\eta^{\frac{\rho_4}{2}}\big)
+a\big(\eta+\eta^{\frac{1}{2}}\big)
\Big]
\|v-\tilde{v}\|_{Y,t_0+\eta,t_0+2\eta}\\
&<\xi\|v-\tilde{v}\|_{Y,t_0+\eta,t_0+2\eta}
\end{align*} 
holds. This implies $v=\tilde{v}$ 
on $[t_0+\eta,t_0+2\eta].$ 
Repeating this procedure, we conclude $v=\tilde{v}$. 
\end{proof}

\begin{rmk}\label{rmkunique}
It is clear that the equation 
(\ref{bint}) admits at most one solution 
within $Y_{0}$ under the same condition 
as in Lemma \ref{yr0uniqueness}.
\end{rmk}
Let us close the paper with completion 
of the proof of Theorem \ref{attainthm}.\vspace{0.4cm}
\par\noindent{\bf Proof of Theorem \ref{attainthm}}\quad 
Since we know $\varepsilon\leq \tilde{\varepsilon}$ 
from (\ref{assump}) and (\ref{assump2}),
Proposition \ref{yr0existence} and 
Lemma \ref{yr0uniqueness} yield the unique existence of solutions 
in $Y_{0}$ when $(M+1)a^{(n-2)/(n+1)}<\varepsilon$. 
Moreover, Proposition \ref{propsharp3} and 
Proposition \ref{propsharp} give us sharp decay properties of the solution 
provided $a$ is still smaller.  
We finally show 
the uniqueness of the solution constructed above within $Y$
by following the argument due to Brezis \cite{brezis1994}. 
It suffices to show that if $v\in Y$ is a solution, it necessarily 
satisfies 
\begin{align}\label{v0}
\lim_{t\rightarrow 0}\,[v]_t=0,
\end{align}
where
\begin{align*}
[v]_t:=\sup_{0<\tau<t}\tau^{\frac{1}{2}}(\|v(\tau)\|_\infty
+\|\nabla v(\tau)\|_n). 
\end{align*}
We assume 
\begin{align}\label{hatepsilon}
(M+1)a^\frac{n-2}{n+1}<\min\left\{\delta,
\frac{1}{2C_2},\frac{1}{16C_1C_0},\frac{1}{16C_1C_3}\right\}=:
\hat{\varepsilon}(n,D)\,(\leq \varepsilon)
\end{align}
and let $v\in Y$ be a solution. 
Here, the constants $C_i$ are as in Remark \ref{rmkexist} 
as well as in the proof of Proposition \ref{yr0existence}. 
Since $v\in BC([0,\infty);L_\sigma^n(D))$ with $v(0)=0$, 
there exists $s_0>0$ 
such that 
\begin{align*}
\|v(s)\|_n+(M+1)a^\frac{n-2}{n+1}<\hat{\varepsilon}
\end{align*}
for all $0<s\leq s_0$. Hence by Remark \ref{rmkexist}, 
the integral equation (\ref{bint}) with $b=v(s)$ admits a solution 
$T(t)v(s)\in Y_{0}$ along with 
\begin{align}\label{vsest}
\|T(\cdot)v(s)\|_{Y}\leq
4\big(C_0\|v(s)\|_n+C_3(M+1)a^\frac{n-2}{n+1}\big)
<4(C_0+C_3)\hat{\varepsilon}\leq \frac{1}{2C_1}.
\end{align}
On the other hand, given $s\in (0,s_0]$, we can see that  
$z_s(t):=v(t+s)$ for $t\geq 0$ also 
satisfies (\ref{bint}) with $b=v(s)$ and $z_s\in Y_{0}$. 
In view of Remark \ref{rmkunique}, 
we have $z_s(t)=T(t)v(s)$ for $s\in(0,s_0]$, which implies 
\begin{align*}
t^{\frac{1}{2}}\big(\|v(t+s)\|_\infty+\|\nabla v(t+s)\|_n\big)
\leq \displaystyle\sup_{f\in K}[T(\cdot)f]_t,\quad
K=v((0,s_0]):=\{v(t)\in L^n_\sigma(D)\mid t\in(0,s_0]\} 
\end{align*}
for all $s\in (0,s_0]$ and $t>0$. 
Passing to the limit $s\rightarrow 0$, we find
\begin{align}\label{ttt}
[v(\cdot)]_t\leq\displaystyle\sup_{f\in K}[T(\cdot)f]_t.
\end{align}
Furthermore, 
applying Lemma \ref{key} to (\ref{bint}) with $b=f\in v((0,s_0])$ 
as well as Proposition \ref{proplqlr} and (\ref{usest}), 
we have 
\begin{align*}
[T(\cdot)f]_t&\leq C_0[S(\cdot)f]_t
+\Big(C_1\sup_{f\in K}\|T(\cdot)f\|_{Y}
+C_2a^\frac{n-2}{n+1}\Big)[T(\cdot)f]_t
+\|H_1\|_{Y,t}+\|H_2\|_{Y,t},
\end{align*}
where $S(t)f:=e^{-tA_a}f$, and 
deduce from (\ref{hatepsilon}) and (\ref{vsest}) that 
\begin{align}\label{Tf}
[T(\cdot)f]_t\leq 
\frac{C_0[S(\cdot)f]_t+
\|H_1\|_{Y,t}+\|H_2\|_{Y,t}}{1-\Big(C_1\displaystyle
\sup_{f\in K}\|T(\cdot)f\|_{Y}
+C_2a^\frac{n-2}{n+1}\Big)}
\end{align}
for all $f\in K$ and $t>0$. 
Collecting (\ref{Tf}), (\ref{et0k}), (\ref{ttt}) 
and $H_1,H_2\in Y_{0}$ leads to (\ref{v0}). 
The proof is complete. \qed
\vspace{0.4cm}\\
{\bf Acknowledgment.}~
The author would like to 
thank Professor Toshiaki Hishida for valuable 
comments and constant encouragement. 
The author is also grateful to the referee for suggesting 
the another way to prove Theorem \ref{rmkrho}, which is mentioned in 
Section \ref{section1}.

\begin{align*}
\hspace{8cm}&\\
&\rm Graduate~School~of~Mathematics\\
&\rm Nagoya ~University\\
&\rm Furo\mbox{-}cho, Chikusa\mbox{-}ku\\
&\rm Nagoya, 464\mbox{-}8602 \\
&\rm Japan\\
&\rm E\mbox{-}mail: m17023c@math.nagoya\mbox{-}u.ac.jp
\end{align*}


\begin{thebibliography}{99}
\bibitem{adams1975}
Adams, R.A.: Sobolev Spaces. Academic Press, New York (1975)

\bibitem{babenko1973}
Babenko, K.I.: On stationary solutions of the problem of flow 
past a body of viscous incompressible fluid. 
Math. Sb. {\bf91}, 3--27 (1973); English Translation: 
Math. USSR Sbornik {\bf20}, 1--25 (1973)



\bibitem{bogovskii1979}
Bogovski\u{\i}, M. E.: Solution of the first boundary value problem 
for the equation of continuity of an incompressible medium. 
Sov. Math. Dokl. {\bf{20}}, 1094--1098 (1979)



\bibitem{bosh1990}
Borchers, W., Sohr, H.: On the equations $\rot$v=g and $\div$u=f 
with zero boundary conditons. Hokkaido Math. 
J. {\bf{19}}, 67--87 (1990)


\bibitem{brezis1994}
Brezis, H.: Remarks on the preceding paper by M. Ben-Artzi, 
$\textquoteleft\textquoteleft$Global solutions of two 
dimensional Navier-Stokes and Euler equations$\textquotedblright$. 
Arch. Ration. Mech. Anal. {\bf{128}}, 359--360 (1994)

\bibitem{chen1993}
Chen, Z.M.: Solutions of the stationary and nonstationary Navier-Stokes 
equations in exterior domains. Pacific J. Math. {\bf 159}, 227--240
(1993)



\bibitem{ensh2004}
Enomoto, Y., Shibata, Y.: 
Local Energy Decay of Solutions to the Oseen Equation 
in the Exterior Domains.
Indiana Univ. Math. J. {\bf{53}}, 1291--1330 (2004)



\bibitem{ensh2005}
Enomoto, Y., Shibata, Y.: 
On the Rate of Decay of the Oseen Semigroup in the Exterior Domains 
and its Application to Navier-Stokes Equation. 
J. Math. fluid mech. {\bf{7}}, 339--367 (2005)



\bibitem{farwig1992t}
Farwig, R.:
The stationary exterior 3D-problem of Oseen and Navier-Stokes equations 
in anisotropically weighted Sobolev spaces. 
Math. Z. {\bf{211}}, 409--447 (1992)



\bibitem{faso1998}
Farwig, R., Sohr, H.: Weighted estimates for the Oseen equations 
and the Navier-Stokes equations in exterior domains.  
In: Heywood, J.G. et al. (eds.) 
Theory of the Navier-Stokes Equations, 
Series on Advances in Mathematics for Applied Sciences, vol. 47, pp. 
11--30. World Scientific Publishing, 
River Edge (1998) 


\bibitem{finn1959}
Finn, R.: Estimate at infinity for stationary solutions 
of the Navier-Stokes equations. Bull. Math. Soc. Sci. 
Math. Phys. R. P. Roumanie (N.S.) {\bf{3}}, 387--418 (1959)




\bibitem{finn1965s}
Finn, R.: Stationary solutions of the Navier-Stokes equations. 
In: Proceedings of Symposia in Applied Mathematics, vol. 17, 
pp. 121--153 (1965)






\bibitem{finn1965o}
Finn, R.: On the exterior stationary problem for 
the Navier-Stokes equations, 
and associated perturbation problems. 
Arch. Rational Mech. Anal. {\bf{19}}, 363--406 (1965)



\bibitem{finn1973}
Finn, R.: Mathematical questions relating to viscous fluid flow 
in an exterior domain. 
Rocky Mountain J. Math. {\bf{3}}, 107--140 (1973)



\bibitem{fuka1964}
Fujita, H., Kato, T.: On the Navier-Stokes initial value problem. I. 
Arch. Rational Mech. Anal. {\bf{16}}, 269--315 (1964)




\bibitem{fumo1977}
Fujiwara, D., Morimoto, H.: 
An $L_r$-theorem of the Helmholtz decomposition of vector fields. 
J. Fac. Sci. Univ. Tokyo Sect. IA Math {\bf{24}}, 685--700 (1977)


\bibitem{galdi1992oseen}
Galdi, G.P.: On the Oseen boundary value problem in exterior domains. 
In: The Navier-Stokes Equations II --- Theory 
and Numerical Methods (Oberwolfach, 1991), vol. 1530 of Lecture Notes
in Math., pp. 111--131. Springer, Berlin (1992)


\bibitem{galdi1992}
Galdi, G.P.: On the Asymptotic Structure of $D$-Solutions to 
Steady Navier-Stokes Equations in Exterior Domains.  
In: Galdi, G.P. (ed.) 
Mathematical Problems Related to the Navier-Stokes Equation, 
Series on Advances in Mathematics for Applied Sciences, vol. 11, 
pp. 81--104. World Scientific Publishing, River Edge (1992)  


\bibitem{galdi2011}
Galdi, G.P.: An Introduction to the Mathematical Theory 
of the Navier-Stokes Equations.  
Steady-State Problems, Second Edition. Springer, New York (2011)



\bibitem{gahesh1997}
Galdi, G.P., Heywood, J.G., Shibata, Y.: 
On the global existence and convergence to steady state of 
Navier-Stokes flow past an obstacle that is started from rest. 
Arch. Rational Mech. Anal. {\bf{138}}, 307--318 (1997)



\bibitem{heywood1972}
Heywood, J.G.: The exterior nonstationary problem for 
the Navier-Stokes equations. Acta Math. 
{\bf{129}}, 11--34 (1972)

\bibitem{hishida2016}
Hishida, T.: $L^q$-$L^r$ estimate of Oseen flow in plane exterior domains.
J. Math. Soc. Japan {\bf{68}}, 295--346 (2016)



\bibitem{hima2018}
Hishida, T., Maremonti, P.:
Navier-Stokes flow past a rigid body: attainability 
of steady solutions as limits of unsteady weak solutions, 
starting and landing cases. 
J. Math. Fluid Mech. {\bf{20}}, 771--800 (2018)





\bibitem{kato1984}
Kato, T.: Strong $L^p$ solutions of the Navier-Stokes equation in $\R^m$, 
with applications to weak solutions. 
Math. Z. {\bf{187}}, 471--480 (1984)



\bibitem{koba2017}
Koba, H.: On $L^{3,\infty}$-stability of the Navier-Stokes system 
in exterior domains. J. Differential Equations 
{\bf{262}}, 2618--2683 (2017)


\bibitem{kosh1998}
Kobayashi, T., Shibata, Y.: On the Oseen equation 
in the three dimensional exterior domains. 
Math. Ann. {\bf{310}}, 1--45 (1998)

\bibitem{leray1933}
Leray, J.: \'{E}tude de diverses \'{e}quations int\'{e}grales 
non lin\'{e}aires 
et de quelques probl\`{e}mes que pose l'hydrodynamique. 
J. Math. Pures Appl. {\bf 12}, 1--82 (1933)

\bibitem{maekawa2019loc}
Maekawa, Y.: On local energy decay estimate of the Oseen semigroup in two
dimensions and its application. J. Inst. Math. Jussieu (2019), 
https://doi.org/10.1017/s1474748019000355

\bibitem{maekawa2019stability}
Maekawa, Y.: On stability of physically reasonable solutions 
to the two-dimensional Navier-Stokes equations. 
J. Inst. Math. Jussieu (2019), https://doi.org/10.1017/s1474748019000240


\bibitem{miyakawa1982}
Miyakawa, T.:
On nonstationary solutions of the Navier-Stokes equations 
in an exterior domain. 
Hiroshima Math. J. {\bf{12}}, 115--140 (1982)



\bibitem{shibata1999}
Shibata, Y.: On an exterior initial boundary value problem 
for Navier-Stokes equations. 
Quart. Appl. Math. {\bf{LVII}}, 117--155 (1999)


\bibitem{shya2005}
Shibata, Y., Yamazaki, M.: Uniform estimates in the velocity 
at infinity for stationary solutions 
to the Navier-Stokes exterior problem. 
Japan. J. Math. (N.S.) {\bf{31}}, 225--279 (2005)


\bibitem{siso1992}
Simader, C.G., Sohr, H.: A new approach to the Helmholtz decomposition and 
the Neumann problem in $L^q$-spaces for bounded and exterior domains. 
In: Galdi, G.P. (ed.) 
Mathematical Problems Relating to the Navier-Stokes Equations,  
Series on Advances in Mathematics for Applied Sciences, vol. 11,
pp. 1--35. World Scientific Publishing, River Edge (1992)

\bibitem{tomopre}
Takahashi, T.: Attainability of a stationary Navier-Stokes 
flow around a rigid body rotating from rest. arXiv:2004.00781,
to appear in Funkcial. Ekvac.
\end{thebibliography}
\end{document}